\documentclass[a4paper,12pt]{amsart}
\usepackage[english]{babel}
\usepackage{amsmath,amsthm,amssymb,amsfonts}
\usepackage{multicol}
 
\usepackage[pdftex]{color}
\usepackage[bookmarks=true,hyperindex,pdftex,colorlinks, citecolor=blue,linkcolor=blue, urlcolor=blue]{hyperref}
\usepackage[dvipsnames]{xcolor}
\usepackage{mathtools}
\usepackage{graphicx}
\usepackage[titletoc]{appendix}
\usepackage{tikz}

\usepackage[normalem]{ulem}

\usetikzlibrary{matrix}

\newtheorem{theorem}{Theorem}[section]
\newtheorem{lemma}[theorem]{Lemma}

\newtheorem{proposition}[theorem]{Proposition}
\newtheorem{corollary}[theorem]{Corollary}
\newtheorem{question}[theorem]{Question}
\theoremstyle{definition}
\newtheorem{definition}[theorem]{Definition}
\newtheorem{example}[theorem]{Example}

\theoremstyle{remark}
\newtheorem{remark}[theorem]{Remark}
\numberwithin{equation}{section}

\newcommand{\R}{\mathbb{R}}
\newcommand{\C}{\mathbb{C}}
\newcommand{\N}{\mathbb{N}}
\newcommand{\K}{\mathbb{K}}

\DeclareMathOperator{\dist}{dist\,}

\DeclareMathOperator{\co}{co}

\DeclareMathOperator{\re}{Re}
\DeclareMathOperator{\Ext}{Ext}
\DeclareMathOperator{\id}{Id}
\newcommand{\cconv}{\overline{\co}}

\newcommand{\nn}[1]{{\left\vert\kern-0.25ex\left\vert\kern-0.25ex\left\vert #1 
		\right\vert\kern-0.25ex\right\vert\kern-0.25ex\right\vert}}

\renewcommand{\geq}{\geqslant}
\renewcommand{\leq}{\leqslant}

\newcommand{\NA}{\operatorname{NA}}
\newcommand{\spann}{\operatorname{span}}

\newcommand{\pten}{\ensuremath{\widehat{\otimes}_\pi}}
\newcommand{\iten}{\ensuremath{\widehat{\otimes}_\varepsilon}}
\newcommand{\eps}{\varepsilon}

\begin{document}

\title{Norm-attaining tensors and nuclear operators}

\author[Dantas]{Sheldon Dantas}
\address[Dantas]{Departament de Matemàtiques and Institut Universitari de Matemàtiques i Aplicacions de Castelló (IMAC), Universitat Jaume I, Campus del Riu Sec. s/n, 12071 Castelló, Spain. \newline
	\href{http://orcid.org/0000-0001-8117-3760}{ORCID: \texttt{0000-0001-8117-3760} } }
\email{\texttt{dantas@uji.es}}
\urladdr{\url{www.sheldondantas.com}}

\author[Jung]{Mingu Jung}
\address[Jung]{Department of Mathematics, POSTECH, Pohang 790-784, Republic of Korea \newline
\href{http://orcid.org/0000-0000-0000-0000}{ORCID: \texttt{0000-0003-2240-2855} }}
\email{\texttt{jmingoo@postech.ac.kr}}

\author[Rold\'an]{\'Oscar Rold\'an}
\address[Rold\'an]{Departamento de An\'alisis Matem\'atico, Universitat de Val\`encia, Valencia, Spain \newline
\href{https://orcid.org/0000-0002-1966-1330}{ORCID: \texttt{0000-0002-1966-1330} }}
\email{\texttt{oscar.roldan@uv.es}}

\author[Rueda Zoca]{Abraham Rueda Zoca}
\address[Rueda Zoca]{Universidad de Murcia, Departamento de Matem\'{a}ticas, Campus de Espinardo 30100 Murcia, Spain
 \newline
\href{https://orcid.org/0000-0003-0718-1353}{ORCID: \texttt{0000-0003-0718-1353} }}
\email{\texttt{abraham.rueda@um.es}}
\urladdr{\url{https://arzenglish.wordpress.com}}

\begin{abstract} Given two Banach spaces $X$ and $Y$, we introduce and study a concept of norm-attainment in the space of nuclear operators $\mathcal N(X,Y)$ and in the projective tensor product space $X\pten Y$. We exhibit positive and negative examples where both previous norm-attainment hold. We also study the problem of whether the class of elements which attain their norms in $\mathcal N(X,Y)$ and in $X\pten Y$ is dense or not. We prove that, for both concepts, the density of norm-attaining elements holds for a large class of Banach spaces $X$ and $Y$ which, in particular, covers all classical Banach spaces. Nevertheless, we present Banach spaces $X$ and $Y$ failing the approximation property in such a  way that the class of elements in $X\pten Y$ which attain their projective norms is not dense. We also discuss some relations and applications of our work to the classical theory of norm-attaining operators throughout the paper.
\end{abstract}

\thanks{The first author was supported by the project OPVVV CAAS CZ.02.1.01/0.0/0.0/16\_019/0000778 and by the Estonian Research Council grant PRG877. The second author was supported by the Basic Science Research Program through the National Research Foundation of Korea (NRF) funded by the Ministry of Education (NRF-2019R1A2C1003857). The third author was supported by the Spanish Ministerio de Ciencia, Innovaci\'on y Universidades, grant FPU17/02023, and by the MINECO and FEDER project MTM2017-83262-C2-1-P. The fourth author was supported by Juan de la Cierva-Formaci\'on fellowship FJC2019-039973, by MTM2017-86182-P (Government of Spain, AEI/FEDER, EU), by MICINN (Spain) Grant PGC2018-093794-B-I00 (MCIU, AEI, FEDER, UE), by Fundaci\'on S\'eneca, ACyT Regi\'on de Murcia grant 20797/PI/18, by Junta de Andaluc\'ia Grant A-FQM-484-UGR18 and by Junta de Andaluc\'ia Grant FQM-0185.}


\maketitle

\thispagestyle{plain}

\section{Introduction}

One of the most classical topics in the theory of Banach spaces is the study of norm-attaining functions. As a matter of fact, one of the most famous characterizations of reflexivity, due to R. James, is described in terms of linear functionals which attain their norms (see, for instance, \cite[Corollary 3.56]{FHHMPZ}). In the same direction, E. Bishop and R. Phelps proved that the set of all norm-attaining linear functionals is dense in $X^*$ (see \cite{BP}). This motivated J. Lindenstrauss to study the analogous problem for bounded linear operators in his seminal paper \cite{linds2}, where it was obtained for the first time an example of a Banach space such that the Bishop-Phelps theorem is no longer true for this class of functions. Consequently, this opened the gate for a crucial and vast research on the topic during the past fifty years in many different directions. Indeed, just to name a few, J. Bourgain, R.E. Huff, J. Johnson, W. Schachermayer, J.J. Uhl, J. Wolfe, and V. Zizler continued the study about the set of all linear operators which attain their norms (\cite{B, H, JW, S, U, Z}); M. Acosta, R. Aron, F.J. Aguirre, Y.S. Choi, R. Pay\'a (\cite{gasp, AFW, Choi} tackled problems in the same line involving bilinear mappings; D. Garc\'ia and M. Maestre considered it for homogeneous polynomials (see \cite{AGM, ADM}); and more recently several problems on norm-attainment of Lipschitz maps were considered (see \cite{CCGMR,CGMR,godsurvey,kms}).

Six years ago, M. Mart\'in solved negatively a problem from the 1970s (posed explicitly by J. Diestel and J. Uhl in \cite{DU} and J. Johnson and J. Wolfe in \cite{JW}) on whether or not every compact operator can be approximated by norm-attaining operators (see \cite[Theorem 1]{martinjfa}). On the other hand, the main open problem in the theory of norm-attaining operators nowadays seems to be if every finite-rank operator can be approximated by norm-attaining operators (see \cite[Question 9]{martinjfa}). Since every nuclear operator is a limit of a sequence of finite-rank operators, we were motivated to try to take one step further in the theory by studying the set of all nuclear operators which attain their (nuclear) norms systematically.

On account of clear relations between nuclear operators and projective tensor products, we focus also on a concept of norm-attainment in projective tensor products (see Definition \ref{defi:normattainnuc}). This is justifiable, since it has strong and deep connections with different open problems coming from the study of norm-attaining operators. To mention one of them, let $Y$ be a finite dimensional Banach space. Then, for an arbitrary Banach space $Z$, every operator from $Y$ into $Z$ attains its norm by using the compactness of the unit ball of $Y$. If we suppose that the same happens with the nuclear operators, since $Y$ is finite dimensional, we would have that the set of all norm-attaining tensors in $Y^* \pten Z$ is the whole set $Y^* \pten Z$ for every Banach space $Z$. By Corollary \ref{cor:projnotnormattain}, the set of all norm-attaining operators from $Z$ into $Y$ would be dense in $\mathcal{L}(Z, Y)$ for every Banach space $Z$ and this would mean finally that every finite-rank operator can be approximated by norm-attaining operators.

We proceed now to describe the content of the paper. In Section \ref{back}, we give the necessary background material to help the reader to follow the track of ideas from the text without having to jump into references so often. In particular, we give the precise definitions of norm-attainment in the context of nuclear operators and tensor products (see Subsection \ref{back:normattaining}) as well as the concepts of approximations. Section \ref{section:norm-attaining} is devoted to the first examples of nuclear operators and tensors which attain their norms. We give a characterization for these kind of elements, which will be very helpful during the entire paper. We prove that if every element in the projective tensor product between two Banach spaces $X, Y$ attains its projective norm, then the set of all norm-attaining operators from $X$ into $Y^*$ is dense. Since there exist operators which cannot be approximated by norm-attaining operators, this result gives the first examples of nuclear operators that do not attain their nuclear norms, meaning that the study of norm-attaining nuclear operators is not a trivial problem. In Section \ref{section:denseness}, we show that the set of all norm-attaining tensors (in particular, we get the analogous result for nuclear operators) is dense in the projective tensor product whenever both factors are finite dimensional Banach spaces (actually, our result is more general than this). 
By using this result and the fact that the projective norm respects 1-complemented subspaces, we prove that the density problem holds in a much more general scenario. Indeed, we prove that if the involved Banach spaces satisfy a property which guarantees the existence of many 1-complemented subspaces (see Definition \ref{def:propertyP}), then every tensor can be approximated (in the projective norm) by norm-attaning tensors (and the result for nuclear operators follows as a particular case). Since this property is satisfied by Banach spaces with finite dimensional decompositions of constant $1$, $L_p$-spaces, and $L_1$-predual spaces, the problem of denseness for nuclear operators and tensors is covered by all classical Banach spaces. Moreover, we prove that such a property is stable by finite absolute sums, countable $c_0$- and $\ell_p$-sums, projective tensor products, and injective tensor products. In Section \ref{contraejemplocojonudo}, we present an example of two Banach spaces $X$ and $Y$, both failing the approximation property, which shows that the set of norm-attaining tensors is not always dense in the projective tensor product space based on the counterexample given in \cite{martinjfa} with the existence of an equivalent renorming of $c_0$ which has bidual strictly convex (see \cite{KGM, KGMW, R}). Finally, we finish the paper with a discussion on some open problems.

\section{Background, Notation, and Concepts} \label{back}

\subsection{Basic Notation} We use essentially the notation from \cite{rya}. Let $X$, $Y$, and $Z$ be Banach spaces over the field $\K$, which can be either $\R$ or $\C$. We denote by $B_X$ and $S_X$ the closed unit ball and the unit sphere, respectively, of the Banach space $X$. We denote by $\mathcal{L}(X, Y)$ the set of all bounded linear operators from $X$ into $Y$. If $Y = \K$, then $\mathcal{L}(X, \K)$ is denoted by $X^*$, the topological dual space of $X$. We denote by $\mathcal{B}(X \times Y, Z)$ the Banach space of bounded bilinear mappings from $X \times Y$ into $Z$. When $Z = \K$, we denote this space by $\mathcal{B}(X \times Y)$. It is well-known that the space $\mathcal{B}(X \times Y)$ and $\mathcal{L}(X, Y^*)$ are isometrically isomorphic as Banach spaces. We denote by $\mathcal{K}(X, Y)$ the set of all compact operators and by $\mathcal{F}(X, Y)$ the space of all operators of finite-rank from $X$ into $Y$. Given an absolute norm $| \cdot |_a$ defined on $\mathbb{R}^2$, let us denote by $X \oplus_a Y$ the {\it absolute sum} of $X$ and $Y$ with respect to $| \cdot |_a$, which is a Banach space $X \times Y$ endowed with the norm $\|(x, y)\|_a = |(\|x\|, \|y\|)|_a$ for every $x \in X$ and $y \in Y$.

\subsection{Tensor Products and Nuclear Operators} \label{subsect:tensorprod}

The projective tensor product of $X$ and $Y$, denoted by $X \pten Y$, is the completion of the space $X \otimes Y$ endowed with the norm given by 

\begin{eqnarray*}
\|z\|_{\pi} &=& \inf \left\{ \sum_{n=1}^{\infty} \|x_n\| \|y_n\|: \sum_{n=1}^{\infty} \|x_n\| \|y_n\| < \infty, z = \sum_{n=1}^{\infty} x_n \otimes y_n \right\} \\
&=& \inf \left\{ \sum_{n=1}^{\infty} |\lambda_n|: z = \sum_{n=1}^{\infty} \lambda_n x_n \otimes y_n, \sum_{n=1}^{\infty} |\lambda_n| < \infty, \|x_n\| = \|y_n\| = 1 \right\},
\end{eqnarray*}
where the infinum is taken over all such representations of $z$. It is well-known that $\|x \otimes y\|_{\pi} = \|x\| \|y\|$ for every $x \in X$, $y \in Y$, and the closed unit ball of $X \pten Y$ is the closed convex hull of the set $B_X \otimes B_Y = \{ x \otimes y: x \in B_X, y \in B_Y \}$. Throughout the paper, we will be using both formulas indistinctly, without any explicit reference. The canonical identification $\mathcal{B}(X \times Y, Z) = \mathcal{L}(X \pten Y, Z)$ allows us to obtain the canonical identification $\mathcal{B}(X \times Y) = (X \pten Y)^*$. Using the fact that the spaces $\mathcal{B}(X \times Y)$ and $\mathcal{L}(X, Y^*)$ are isometrically isomorphic, we also have the identification $(X \pten Y)^* = \mathcal{L}(X, Y^*)$, where the action of an operator $G: X \longrightarrow Y^*$ as a linear functional on $X \pten Y$ is given by
\begin{equation*}
G \left( \sum_{n=1}^{\infty} x_n \otimes y_n \right) = \sum_{n=1}^{\infty} G(x_n)(y_n),
\end{equation*}
for every $\sum_{n=1}^{\infty} x_n \otimes y_n \in X \pten Y$. Let us recall also that there is a canonical operator $J: X^* \pten Y \longrightarrow \mathcal{L}(X, Y)$ with $\|J\| = 1$ defined by $z = \sum_{n=1}^{\infty} \varphi_n \otimes y_n \mapsto L_z$, where $L_z: X \longrightarrow Y$ is given by
\begin{equation*} 
L_z(x) = \sum_{n=1}^{\infty} \varphi_n(x) y_n \ \ \ (x \in X).
\end{equation*} 
The operators that arise in this way are called {\it nuclear operators}. We denote the set of such operators by $\mathcal{N}(X, Y)$ endowed with the {\it nuclear norm} 
\begin{equation*}
\|T\|_N = \inf \left\{ \sum_{n=1}^{\infty} \|x_n^*\| \|y_n\|: T(x) = \sum_{n=1}^{\infty} x_n^*(x) y_n \right\},
\end{equation*}
where the infimum is taken over all representations of $T$ of the form $T(x) = \sum_{n=1}^{\infty} x_n^*(x) y_n$ for bounded sequences $(x_n^*) \subseteq X^*$ and $(y_n) \subseteq Y$ such that $\sum_{n=1}^{\infty} \|x_n^*\| \|y_n\| < \infty$. Notice that every nuclear operator is compact since it is the limit in the operator norm of a sequence of finite-rank operators. Using the function $J$, we can identify the space $\mathcal{N} (X, Y)$ with $X^* \pten Y / \ker J$ isometrically. In order to clarify the relations between the set of nuclear operators, the quotient space of the projective tensor product and their respective duals, we consider the following diagram:
	\[
	\begin{tikzpicture} \label{remark:identinucleproj}
	\matrix (m) [matrix of math nodes,row sep=3em,column sep=4em,minimum width=2em]
	{
		(\ker J)^\perp & \left( X^* \pten Y / \ker J \right)^* & \mathcal{N}(X,Y)^* \\
		\, & X^* \pten Y / \ker J & \mathcal{N}(X,Y) \\ };
	\path[-stealth]
	(m-1-1) edge node [above] {$\delta$} (m-1-2)   
	(m-1-2) edge [dashed,-] (m-2-2)
	(m-1-3) edge node [above] {$\widetilde{J}^*$} (m-1-2)
	(m-2-2) edge node [above] {$\widetilde{J}$} (m-2-3)
	(m-1-3) edge [dashed, -] (m-2-3);
	\end{tikzpicture},
	\]
	where $\widetilde{J}$ and $\delta$ are isometric isomorphisms between $X^* \pten Y / \ker J$ and $\mathcal{N}(X, Y)$, and $(\ker J)^\perp$ and $\left( X^* \pten Y / \ker J \right)^*$, respectively.
If we consider a nuclear operator $T \in \mathcal{N}( X, Y)$ given by $T = \sum_{n=1}^\infty x_n^* \otimes y_n$ for some $(x_n^*)_{n \in \N} \subset X^*$ and $(y_n)_{ n \in \N} \subset Y$ bounded with $\sum_{n=1}^{\infty} \|x_n^*\| \|y_n\| < \infty$, then for every $H \in \mathcal{N}(X, Y)^*$, we have 
\[
H (T) = \sum_{n=1}^\infty G(x_n^*) (y_n),
\]
where $G = (\delta^{-1} \circ \widetilde{J}^* ) (H) \in (\ker J)^\perp$.

Recall that a Banach space is said to have the {\it approximation property} if for every compact subset $K$ of $X$ and every $\eps > 0$, there exists a finite-rank operator $T: X \longrightarrow X$ such that $\|T(x) - x\| \leq \eps$ for every $x \in K$. Let us take into account that if $X^*$ or $Y$ has the approximation property, then $X^*\pten Y = \mathcal{N}(X, Y)$ (see, for instance, \cite[Corollary 4.8]{rya}). Recall also that the injective norm of $z \in X \otimes Y$ is defined by
\begin{equation*}
\|z\|_{\eps} = \sup \left\{ \left| \sum_{i=1}^n x^*(x_i) y^*(y_i) \right|: x^* \in B_{X^*}, y^* \in B_{Y^*} \right\},
\end{equation*}
where $\sum_{i=1}^n x_i \otimes y_i$ is any representation of $z$. We denote by $X \otimes_{\eps} Y$ the tensor product $X \otimes Y$ with the injective norm and its completion, denoted by $X \iten Y$, is called the {\it injective tensor product} of $X$ and $Y$.

For a complete background on tensor products in Banach spaces, we refer the reader to the books \cite{DF, rya}.

\subsection{Norm-attaining concepts} \label{back:normattaining} Recall that $T \in \mathcal{L}(X, Y)$ {\it attains its norm} (in the classical way) if there is $x_0 \in S_X$ such that $\|T(x_0)\| = \|T\| = \sup_{x \in S_X} \|T(x)\|$. In this case, we say that $T$ is a {\it norm-attaining} operator. Recall also that $B \in \mathcal{B}(X \times Y, Z)$ attains its norm if there is $(x_0, y_0) \in S_X \times S_Y$ such that $\|B(x_0, y_0)\| = \|B\| = \sup_{(x, y) \in S_X \times S_Y} \|B(x, y)\|$. In this case, we say that $B$ is a norm-attaining bilinear mapping. In the next sections, we will be considering the concepts of attainment on the Banach spaces $X \pten Y$ and $\mathcal{N}(X, Y)$. For us, the most natural approach is the following one.

\begin{definition}\label{defi:normattainnuc} Let $X, Y$ be Banach spaces. We say that 
\begin{itemize}
 
	\item[(1)] $z \in X \pten Y$ {\it attains its projective norm} if there is a bounded sequence $(x_n, y_n) \subseteq X \times Y$ with $\sum_{n=1}^{\infty} \|x_n\| \|y_n\| < \infty$ such that $z=\sum_{n=1}^\infty x_n\otimes y_n$ and that $\|z\|_{\pi} = \sum_{n=1}^{\infty} \|x_n\| \|y_n\|$. In this case, we say that $z$ is a {\it norm-attaining tensor}.
	\vspace{0.2cm} 
	
	\item[(2)] $T \in \mathcal{N}(X, Y)$ {\it attains its nuclear norm} if there is a bounded sequence $(x_n^*, y_n) \subseteq X^* \times Y$ with $\sum_{n=1}^{\infty} \|x_n^*\|\|y_n\| < \infty$ such that $T=\sum_{n=1}^\infty x_n^*\otimes y_n$ and that $\|T\|_N = \sum_{n=1}^{\infty} \|x_n^*\| \|y_n\|$. In this case, we say that $T$ is a {\it norm-attaining nuclear operator}.
\end{itemize}
\end{definition}

If (1) (respectively, (2)) holds, then we say that $\sum_{n=1}^{\infty} x_n \otimes y_n$ (respectively, $\sum_{n=1}^{\infty} x_n^* \otimes y_n$) is a {\it norm-attaining representation}. Let us fix the notation for the set of norm-attaining operators, bilinear mappings, tensors, and nuclear operators. For the first two, we continue using the classical notation $\NA(X, Y) = \{ T \in \mathcal{L}(X, Y): T \ \mbox{attains its norm} \}$ and $\NA(X \times Y, Z) = \{B \in \mathcal{B}(X \times Y, Z): B \ \mbox{attains its norm} \}$, respectively; if $Z = \K$, then we simply denote it as $\NA(X \times Y)$. For the last two, we shall use the following notations:
\begin{itemize}
 	\item[(1')] $\NA_{\pi} (X \pten Y) = \{ z \in X \pten Y: z \ \mbox{attains its projective norm} \}$. 
	\vspace{0.2cm} 
	\item[(2')] $\NA_{\mathcal{N}} (X, Y) = \{ T \in \mathcal{N}(X, Y): T \ \mbox{attains its nuclear norm} \}$. 
\end{itemize}
\noindent 
Notice that, as we have pointed out before, when $X^*$ or $Y$ has the approximation property then $X^*\pten Y$ is isometrically isomorphic to $\mathcal N(X,Y)$. In such case, it is clear that both concepts of norm-attainment agree. Due to the connection between projective tensor products, bilinear mappings, and operators, we are forced to observe also that the denseness of the sets $\NA(X \times Y)$ and $\NA(X, Y^*)$ are {\it not} equivalent in general, but the first implies the later.

Let us finish this introduction by clarifying what we mean by approximating elements from $X \pten Y$ or $\mathcal{N}(X, Y)$ by norm-attaining ones. Evidently, when working with $X \pten Y$, it is natural to make the approximation of an element $z \in X \pten Y$ by an element $z' \in \NA_{\pi}(X \pten Y)$ using the tensor norm $\|\cdot\|_{\pi}$. Similarly, we shall be dealing with the nuclear operator norm $\|\cdot\|_N$ whenever we approximate a given nuclear operator $T$ by a norm-attaining nuclear operator $T'$.

\section{Nuclear operators and tensors which attain their norms} \label{section:norm-attaining}

In this section, we provide the first examples of elements in $X \pten Y$ and $\mathcal{N}(X, Y)$ which attain their norms. The first result gives us an important characterization used abundantly in the rest of the paper.

\begin{theorem} \label{theo:charprojattain} 
	Let $X, Y$ be Banach spaces. Let $z \in X \pten Y$ with 
	\[
	z = \sum_{n=1}^\infty \lambda_n x_n \otimes y_n,
	\]
	where $\lambda_n \in \mathbb{R}^+$, $x_n \in S_{X}$, and $y_n \in S_Y$ for every $n \in \N$. Then, the following assertions are equivalent: 

\begin{enumerate} 
	\item $\|z\|_{\pi} = \sum_{n=1}^\infty \lambda_n$; in other words, $z\in \NA_{\pi}(X\pten Y)$.
	\vspace{0.1cm}
	\item There is $G \in \mathcal{L} (X, Y^*)$ with $\|G\| = 1$ such that $G(x_n)(y_n) = 1$ for every $n \in \N$. 
	\vspace{0.1cm}
	\item Every norm one $G \in \mathcal{L} (X, Y^*)$ such that $G(z) = \| z \|_\pi $ satisfies that $G(x_n) (y_n) = 1$ for every $n \in \N$. 
\end{enumerate} 
\end{theorem} 

\begin{proof} Suppose that $\|z\|_{\pi} = \sum_{n=1}^\infty \lambda_n$ with $z = \sum_{n=1}^\infty \lambda_n x_n \otimes y_n$ with $ (\lambda_n) \subseteq \mathbb{R}^+$, $(x_n) \subseteq S_X$, and $(y_n) \subseteq S_Y$. Pick any $G \in (X \pten Y)^* = \mathcal{L} (X, Y^*)$ such that $\|G \| = 1$ and $G (z) = \|z\|_{\pi}$. Since we have 
	\[
	\sum_{n=1}^{\infty} \lambda_n = \|z\|_{\pi} = G (z) = \sum_{n=1}^\infty \lambda_n G(x_n) (y_n), 
	\]
	it follows that $G(x_n)(y_n) = 1$ for each $n \in \N$, which proves that (1) implies (3). It is obvious that (3) implies (2). Finally, assume that there exists $G \in \mathcal{L} (X, Y^*)$ with $\|G\| = 1$ such that $G(x_n)(y_n) = 1$ for every $n \in \N$. Then, 
	\[
	\sum_{n=1}^\infty \lambda_n = \sum_{n=1}^\infty \lambda_n G(x_n)(y_n) = G(z) \leq \|z\|_{\pi} \leq \sum_{n=1}^\infty \lambda_n. 
	\]
	This completes the proof.
\end{proof} 

Taking into account the isometric isomorphism between $\mathcal{N}(X,Y)$ and $X^*\pten Y/\ker(J)$, we can take advantage of the previous estimates to prove a nuclear operator version of Theorem \ref{theo:charprojattain} as follows.

\begin{theorem}\label{theo:charnuclattain}
Let $X,Y$ be Banach spaces. Let $T \in \mathcal{N}(X,Y)$ with 
	\[
	T = \sum_{n=1}^\infty \lambda_n x_n^* \otimes y_n, 
	\]
	where $\lambda_n \in \mathbb{R}^+$, $x_n \in S_{X}$, and $y_n \in S_Y$ for every $n \in \N$. Then, the following assertions are equivalent: 
	\begin{enumerate}
		\item $\|T \|_{N} = \sum_{n=1}^\infty \lambda_n $; in other words, $T \in \NA_{\mathcal{N}}(X, Y)$.
			\vspace{0.1cm}
		\item There is $G \in (\ker J)^{\perp}$ with $\|G\| = 1$ such that $G(x_n^*)(y_n) = 1$ for every $n \in \N$.
			\vspace{0.1cm}
		\item For any $G \in (\ker J)^{\perp}$ with $\|G\| = 1$ and $G(T)=\| T \|_N$ we get that $G(x_n^*)(y_n) = 1$ holds for every $n \in \N$. 
	\end{enumerate} 
\end{theorem} 

\begin{proof} Let $\widetilde{J} : X^* \pten Y / \ker J \longrightarrow \mathcal{N}(X, Y)$ be an isometric isomorphism which maps, according to the notation of Subsection \ref{subsect:tensorprod}, $z + \ker J$ to $L_z$. If we let $z_0 := \sum_{n=1}^\infty \lambda_n x_n^* \otimes y_n \in X^* \pten Y$, then $J(z_0) = T$ and $\|T\|_{N} = \| z_0 + \ker J\|$. Now assume (1) and let us prove (3). To this end, pick any $G\in (\ker J)^\perp$ with $\Vert G\Vert=1$ and $G(z_0+\ker J)=\|z_0+\ker J\|$.
Then, 
	\begin{eqnarray*}
	\sum_{n=1}^\infty \lambda_n = \| z_0 + \ker J \| =|G(z_0)| &=& \left|G\left( \sum_{n=1}^\infty \lambda_n x_n^* \otimes y_n \right) \right| \\
	&\leq& \sum_{n=1}^\infty \lambda_n |G(x_n^*)(y_n) | \\
	&\leq& \sum_{n=1}^\infty \lambda_n.
	\end{eqnarray*} 
Then, we have $|G(x_n^*)(y_n)| = 1$ for each $n \in \N$. Using a convexity argument, we get that $G(x_n^*) (y_n) = 1$ for every $n \in \N$. The other implications can be proved as in Theorem \ref{theo:charprojattain}. 
\end{proof} 

With Theorems \ref{theo:charprojattain} and \ref{theo:charnuclattain} in mind, we can now exhibit examples of nuclear operators which attain their nuclear norms.

\begin{example}
Let $X$, $Y$ be two reflexive Banach spaces such that $X^*$ or $Y$ has the approximation property (recall that, in this case, we have $X^*\pten Y = \mathcal{N}(X, Y)$). Assume further that $X^*$ is isometrically isomorphic to a subspace of $Y^*$. Take $G:X^*\longrightarrow Y^*$ to be a linear isometry and pick $(x_n^*)_n \subseteq S_{X^*}$. Now, for any $n\in\mathbb N$, notice that $\Vert G(x_n^*)\Vert=\Vert x_n^*\Vert=1$. Since $Y$ is reflexive, by using the James Theorem, we have that $G(x_n^*)\in S_{Y^*}$ attains its norm, so there exists $y_n\in S_Y$ so that $G(x_n^*)(y_n)=1$. Now, Theorem \ref{theo:charprojattain} (or Theorem \ref{theo:charnuclattain}) implies that, given any sequence $(\lambda_n)_n\subseteq (0, 1]$ with $\sum_{n=1}^\infty \lambda_n<\infty$, the nuclear operator
	$$T:=\sum_{n=1}^\infty \lambda_n x_n^*\otimes y_n \in \mathcal{N}(X, Y)$$
	attains its nuclear norm.
\end{example}

One may think that a norm-attaining nuclear operator should attain its norm (in the classical way). This is not true in general as observed below.

\begin{remark}\label{rema:examnuclearnotnormattain} Let $Y$ be an infinite dimensional strictly convex Banach space. Then, there is $T\in \NA_{\mathcal N}(c_0,Y)$ such that $T \not\in \NA(c_0, Y)$. Indeed, let $(y_n)_n \subseteq S_Y$ be linearly independent. For every $n\in\mathbb N$, find $y_n^*\in S_{Y^*}$ such that $y_n^*(y_n)=1$. Define $\phi:Y\longrightarrow \ell_\infty$ by $\phi(y):=(y_j^*(y))_{j=1}^{\infty} \in \ell_\infty$ for every $y \in Y$. Given $n\in\mathbb N$ we get that $\vert y_n^*(y)\vert\leq \Vert y\Vert$ since $\Vert y_n^*\Vert=1$ holds for every $n\in\mathbb N$. This implies that $\sup_{n\in\mathbb N} \vert y_n^*(y)\vert\leq \Vert y\Vert$, which proves that $\phi(y)\in \ell_\infty$ for every $y$ (i.e., $\phi$ is well defined) . In view of the linearity, we have that $\phi$ is continuous and $\Vert \phi\Vert\leq 1$. Furthermore, notice that $\phi(y_n)(e_n)=1$ holds for every $n\in\mathbb N$, where $(e_n)_{n}$ is the basis of $\ell_1$. This proves that the nuclear operator $T:c_0\longrightarrow Y$ defined by
	$$T=\sum_{n=1}^\infty \frac{1}{2^n}e_n\otimes y_n\in \ell_1\pten Y$$
	attains its nuclear norm by Theorem \ref{theo:charnuclattain}. Nevertheless, notice that $T$ is not a finite-rank operator and, consequently, $T$ does not belong to $\NA(c_0, Y)$ (see \cite[Lemma 2.2]{martinjfa} or the proof of \cite[Proposition 4]{linds2}).
\end{remark}

We prove next that on the the finite dimensional setting, {\it every} tensor is norm-attaining. Before presenting a proof of it, let us notice that since the convex hull of a compact set is compact when $X$ and $Y$ are both finite dimensional spaces, we have that $\cconv{(B_X \otimes B_Y)} = \co{(B_X \otimes B_Y)}$, which is a consequence of Minkowski-Carathéodory theorem (see, for instance, \cite[Exercises 1.57 and 1.58]{FHHMPZ}).

\begin{proposition} \label{fin-dim-tomas} Let $X, Y$ be finite dimensional Banach spaces. Then, every tensor attains its projective tensor norm. In other words, $\NA_{\pi}(X \pten Y) = X \pten Y.$
\end{proposition}

\begin{proof} Let $z \in X \pten Y$ with $\|z\|_{\pi} = 1$ be given and let us prove that $z \in \NA_{\pi}(X \pten Y)$. As we have mentioned before, since $X$ and $Y$ are finite dimensional Banach spaces, $B_X \otimes B_Y$ is compact in $X \pten Y$ and this implies that $B_{X \pten Y} = \overline{\co} (B_X \otimes B_Y) = \co (B_X \otimes B_Y)$. Therefore, $z$ can be written as a finite convex combination of elements in $B_X \otimes B_Y$, i.e.,
	\begin{equation*}
	z = \sum_{j=1}^n \lambda_j x_j \otimes y_j \ \ \ \mbox{with} \ \ \ \sum_{j=1}^n \lambda_j = 1,
	\end{equation*}
	where $\lambda_j \in \R^+$, $x_j \in B_X$, and $y_j \in B_Y$ for $j=1, \ldots, n$, that is, $z$ is norm-attaining.
\end{proof}

Let us notice that in Remark \ref{rema:examnuclearnotnormattain}, we have constructed by hand a nuclear operator from $c_0$ into a particular $Y$ which attains its nuclear norm. It turns out that {\it every} nuclear operator from $c_0$ into {\it any} Banach space $Y$ attains its nuclear norm. This should be compared to the fact that, in the classical theory, whenever $X$ is a Banach space such that $\NA(X, Y) = \mathcal{L}(X, Y)$ for some $Y \not= \{0\}$, $X$ must be reflexive (this is an application of James theorem). In other words, this result is no longer true in the context of nuclear operators.

\begin{proposition}\label{exam:nucopec0attain} Let $Y$ be a Banach space. Then, 
\begin{itemize} 
\item[(a)] every nuclear operator $T\in \mathcal N(c_0, Y)$ attains its nuclear norm. Equivalently, 
\item[(b)] every element in $\ell_1\pten Y$ attains its projective norm. 
\end{itemize} 
\end{proposition}	

\begin{proof} Indeed, in the last part of \cite[Lemma 2.6]{rya}, it is proved that $\Phi: \ell_1(Y)\longrightarrow \ell_1\pten Y$ given by
\begin{equation*}
\Phi((x_n)_n)=\sum_{n=1}^\infty e_n\otimes x_n
\end{equation*}
is an onto linear isometry, where $(e_n)_n$ is the basis of $\ell_1$ (in fact, $\Phi=J^{-1}$ in the proof given there). Let $T\in \mathcal N(c_0,Y)=\ell_1\pten Y$ be given. By the surjectivity of $\Phi$, we can find an element $(x_n)_n \in \ell_1(Y)$ such that $\Phi((x_n)_n)=T$. Consequently, $T=\sum_{n=1}^\infty e_n\otimes x_n$. Then,
\begin{equation*} 
\Vert T\Vert_N=\Vert \Phi((x_n)_n)\Vert=\Vert (x_n)_n\Vert=\sum_{n=1}^\infty \Vert x_n\Vert=\sum_{n=1}^\infty \Vert e_n\Vert \Vert x_n\Vert.
\end{equation*}
This proves that $T$ attains its nuclear norm, as desired.
\end{proof}

\begin{remark}
Notice that Proposition \ref{exam:nucopec0attain} is also true for $c_0(I)$ and $\ell_1(I)$ for any arbitrary index set $I$ (see \cite[Example 2.6]{rya}).
\end{remark}

In the infinite dimensional case, besides the nuclear operators from $c_0$ into an arbitrary Banach space $Y$, we have that every nuclear operator on a complex Hilbert space attains its nuclear norm. Although we prove this result for nuclear operators (justified by the fact that we will be dealing with eigenvalues and Schatten classes), we also get that every tensor in $H\pten H$ attains its projective norm as every Hilbert space $H$ has the approximation property.

\begin{proposition}\label{prop:hilbertnuclearallNA}
Let $H$ be a complex Hilbert space. Then, every nuclear operator $T\in \mathcal N(H, H)$ attains its nuclear norm.
\end{proposition}

\begin{proof}
	Note that $T\in \mathcal{N}(H, H)$ can be written as
\begin{equation*} 
	T = \sum_{j=1}^{n_0} \lambda_j \langle \cdot, x_j\rangle y_j,
\end{equation*} 
where $n_0\in \mathbb{N}\cup \{ \infty \}$, $(\lambda_j)_j$ is the sequence of nonzero eigenvalues of $|T|=(T^* T)^{\frac{1}{2}}$, and $(x_j)_j$ and $(y_j)_j$  are orthonormal systems in $H$ (see \cite[Theorem 2.1]{GGK}). On the other hand, it is well-known that $\| T \|_N = \sigma_1(T)=\sum_{j=1}^{n_0} \lambda_j$, where $\sigma_1(\cdot)$ is the Schatten 1st norm (see, for example, \cite[pages 96-97]{GGK}). This completes the proof.
\end{proof}

Taking into account Propositions \ref{fin-dim-tomas}, \ref{exam:nucopec0attain} and \ref{prop:hilbertnuclearallNA}, it is natural to ask whether or not the equality $\NA_{\mathcal{N}}(X,Y)=\mathcal N(X,Y)$ (or $\NA_{\pi} (X\pten Y)=X\pten Y$) holds for all Banach spaces $X$ and $Y$. We will give a negative answer for this problem by proving that if this happens, then the set of norm-attaining bilinear forms which attain their norms is dense in $\mathcal{B}(X \times Y)$. From our point of view, this shows that the study of norm-attaining nuclear operators is not a trivial task.

\begin{lemma} \label{bilinear-as-functional} Let $X, Y$ be Banach spaces. If $B \in \mathcal{B}(X \times Y) = (X \pten Y)^*$ attains its norm (as a functional) at an element of $\NA_{\pi}(X \pten Y)$, then $B \in \NA(X \times Y)$.
\end{lemma}

\begin{proof} Let $B \in \mathcal{B}(X \times Y) = (X \pten Y)^*$ and $z \in S_{X \pten Y}$ with $z = \sum_{n=1}^{\infty} \lambda_n x_n \otimes y_n \in \NA_{\pi}(X \pten Y)$ be such that $B(z) = 1$, where $\lambda_n \in \R^+$, $x_n \in S_X$, and $y_n \in S_Y$. By Theorem \ref{theo:charprojattain}, $B(x_n, y_n) = 1$ for every $n \in \N$. In particular, $B \in \NA(X \times Y)$.	
\end{proof}

\begin{proposition} \label{prop:allprojattain} Let $X, Y$ be Banach spaces. If every element in $X \pten Y$ attains its projective norm, then the set of all bilinear forms on $X \times Y$ which attain their norms is dense in $\mathcal{B}(X \times Y)$. In other words, if $\NA_{\pi} (X\pten Y)=X\pten Y$, then
\begin{equation*} 
\overline{\NA(X\times Y)}^{\|\cdot\|} = \mathcal{B}(X \times Y).
\end{equation*} 
\end{proposition}

\begin{proof} Let $\varepsilon>0$. Let $B \in \mathcal{B}(X \times Y)=(X\pten Y)^*$ with $\|B\| = 1$. By the Bishop-Phelps theorem, for $X \pten Y$, there are $B_0 \in (X \pten Y)^*$ with $\|B_0\| = 1$ and $z_0 \in S_{X \pten Y}$ such that $B_0(z_0) = 1$ and $\| B_0 - B\| < \eps$. 	By hypothesis, $z_0\in \NA_{\pi}(X,Y)$ attains its projective norm and by Lemma \ref{bilinear-as-functional} we have that $B_0 \in \NA(X \times Y)$. Since $\|B_0 - B\|<\eps$, we are done. 
\end{proof}

Proposition \ref{prop:allprojattain} yields the following consequence.

\begin{corollary}\label{cor:projnotnormattain} Let $X, Y$ be Banach spaces. Suppose that every element in $X \pten Y$ attains its projective norm. Then, the set of norm-attaining operators from $X$ into $Y^*$ is dense in $\mathcal{L}(X, Y^*)$. In other words, if $\NA_{\pi} (X\pten Y)=X\pten Y$, then 
\begin{equation*} 
\overline{\NA(X, Y^* )}^{\|\cdot\|} = \mathcal{L}(X, Y^*).
\end{equation*} 
\end{corollary}

Now, by using Lemma \ref{bilinear-as-functional}, Proposition \ref{prop:allprojattain}, and Corollary \ref{cor:projnotnormattain}, we can get examples of pairs of Banach spaces $(X, Y)$ such that there are elements in the projective tensor product $X \pten Y$ which {\it do not} attain their projective norms.

\begin{example} \label{exa:bilinearnegative} There are elements $z \in X \pten Y$ such that $z \notin \NA_{\pi}(X \pten Y)$ in the following cases. 
	\begin{enumerate}
		
		\item[(a)]\footnote{The authors are thankful to the referee who provided us this example.} When $X = L_1(\mathbb{T})$, where the unit circle $\mathbb{T}$ is equipped with the Haar measure $m$, and $Y$ is the two-dimensional Hilbert space. Indeed, it is shown in \cite[Remark 5.7.(2)]{godsurvey} that there is $T \in \mathcal{B}(X \times Y)$ which attains its norm as a linear functional on $X \pten Y$ but not as an operator from $X$ into $Y^*$ (nor the more as a bilinear form on $X \times Y$). By Lemma \ref{bilinear-as-functional}, it follows that $\NA_{\pi}(X \times Y) \not= X \pten Y$.   
		
		\vspace{0.2cm}
		
		\item[(b)] When $X$ is $L_1[0,1]$ and $Y^*$ is a strictly convex Banach space without the Radon-Nikod\'ym property. Indeed, by \cite[Theorem 3]{U}, the set $\NA(L_1[0,1], Y^*)$ is not dense in $\mathcal{L}(L_1[0,1], Y^*)$.  Let us notice that this also shows that Proposition \ref{exam:nucopec0attain} is no longer true if we consider an $L_1(\mu)$-space for a non-purely atomic measure $\mu$.
		
		\vspace{0.2cm}
		
		\item[(c)] When $Y = \ell_p$ for $1 < p < \infty$ and $X$ is the Banach space constructed by Gowers. Indeed, there is a Banach space $G$ such that $\NA(G \times \ell_p)$ is not dense in $\mathcal{B}(G \times \ell_p)$ (see \cite[Theorem, page 149]{G}). We should notice that the unit ball of $G$ lacks extreme points. This result should be compared to the fact that, if $X$ is reflexive and $Y$ is any Banach space, then $\mathcal{K}(X,Y)\subseteq \NA(X,Y)$.
		
		
		
		\vspace{0.2cm}
		
		\item[(d)] When $X$ and $Y$ are both $L_1[0,1]$. Indeed, \cite[Theorem 3]{Choi} shows that the set $\NA(L_1[0,1] \times L_1[0,1])$ is not dense in $\mathcal{B}(L_1[0,1] \times L_1[0,1])$.
		
		
	\end{enumerate}
\end{example}

Let us finish this section by highlighting two observations. 

\begin{remark}\label{rem:bilinearnegative}
Notice that if we weaken the hypothesis in Proposition \ref{prop:allprojattain} by assuming that $\NA_{\pi}(X\pten Y)$ is dense in $X\pten Y$, the result does {\it not} remain true. Indeed, by using Example \ref{exa:bilinearnegative}.(c), we know that $\NA(L_1[0,1] \times L_1[0,1])$ is not dense in $\mathcal{B}(L_1[0,1] \times L_1[0,1])$, but we will see in Section \ref{section:denseness} that the set of all tensors which attain their projective norm on $L_1[0,1] \pten L_1[0,1]$ is dense in $L_1[0,1] \pten L_1[0,1]$ (see Theorem \ref{theo:propertyP} and Example \ref{examples:propertyP}.(b)). Nevertheless, we will always have that $\operatorname{NA}(X, Y^*)\cap B_{\mathcal{L}(X, Y^*)}$ is $w^*$-dense in $B_{\mathcal{L}(X, Y^*)}$ under this hypothesis (see Remark \ref{remark:NAnormante}).
\end{remark}

\begin{remark} Let $Y$ be a finite dimensional Banach space. Then, $\NA(Y, Z)=\mathcal L(Y, Z)$ for every Banach space $Z$ by using the compactness of the unit ball of $Y$. Let us suppose for a second that the same holds for nuclear operators. Then, $\NA_{\mathcal{N}} (Y, Z) = \mathcal{N}(Y, Z)$ for every Banach space $Z$. Since $Y$ is finite dimensional, it has the approximation property and then we would have that $\NA_{\pi}(Z \pten Y^*) = Z \pten Y^*$ for every Banach space $Z$. By Corollary \ref{cor:projnotnormattain}, we would have that the set $\NA(Z, Y)$ is dense in $\mathcal{L}(Z, Y)$ for every Banach space $Z$, which would imply that $Y$ has property B of Lindenstrauss (solving positively \cite[Question 9]{martinjfa}). Therefore, it is natural to wonder whether every nuclear operator $T:Y\longrightarrow Z$ attain its nuclear norm for every Banach space $Z$ whenever $Y$ is finite dimensional. This is {\it not} the case due to Example \ref{exa:bilinearnegative}.(a)\footnote{It is worth mentioning that this question was posed by the authors in a preliminary version of this manuscript; they thank the anonymous referee who answered it negatively.} by taking $Z = \L_1(\mathbb{T})$ and $Y = \ell_2^2$, the Euclidean plane (see \cite[Remark 5.7.(2)]{godsurvey}).

\end{remark}

\section{Denseness of nuclear operators and tensors which attain their norms} \label{section:denseness} 

Here we will be focusing on examples of Banach spaces $X$ and $Y$ such that the sets $\NA_{\pi}(X \pten Y)$ and $\NA_{\mathcal{N}}(X, Y)$ are dense in norm in $X \pten Y$ and $\mathcal{N}(X, Y)$, respectively. As we have seen in the previous section, there are many examples of projective tensor products where we can guarantee the existence of elements which do not attain their projective norms even when one of the factors is reflexive (see Example \ref{exa:bilinearnegative}.(b)). In spite of the existence of such non-norm-attaining tensors, it is natural to ask if the set of elements in a tensor product space which attain their projective norms is dense in the whole space.

Let us start by explaining where the difficulty comes from when one tries to get such a property. Assume that $z\in \NA_{\pi}(X\pten Y)$ is a norm-attaining tensor in $X \pten Y$. This implies that there are bounded sequences $(x_n)_n\subseteq X$ and $(y_n)_n\subseteq Y$ such that $z=\sum_{n=1}^\infty x_n\otimes y_n$ with $\|z\|_{\pi} =\sum_{n=1}^\infty \Vert x_n\Vert \Vert y_n\Vert$. It is clear that the task of choosing the optimal representation for $z$ as a series of basic tensors is the most difficult part. In order to avoid this inconvenience, let us make use of Theorem \ref{theo:charprojattain}. By applying it, for any bilinear mapping $B \in S_{\mathcal{B}(X \times Y)} = S_{(X\pten Y)^*}$ such that $B(z)=\| z \|_\pi$, we have that $B(x_n)(y_n)=\Vert x_n\Vert \Vert y_n\Vert$ for every $n \in \N$. In other words, $B$ attains its bilinear norm at the pair $\left( \frac{x_n}{\|x_n\|}, \frac{y_n}{\|y_n\|} \right)$ for every $n\in\mathbb N$. Because of this, in order to get examples of Banach spaces $X$ and $Y$ where the set $\NA_{\pi}(X\pten Y)$ is dense in $X \pten Y$, we need somehow that the space $\mathcal B(X\times Y)$ contains many bilinear forms which attain their bilinear norm at many elements of $S_X \times S_Y$. This motivates us to make use of the following definitions, which can be found in \cite{D} and \cite{dklm}.

\begin{definition} \label{defi:Loo} Let $X, Y$ and $Z$ be Banach spaces. 
\begin{itemize}
	\item[(a)] We say that $(X, Y)$ has the  {\bf L}$_{o,o}$ for operators if given $\eps > 0$ and $T \in \mathcal{L}(X, Y)$ with $\|T\| = 1$, there is $\eta(\eps, T) > 0$ such that whenever $x \in S_X$ satisfies $\|T(x)\| > 1 - \eta(\eps, T)$, there is $x_0 \in S_X$ such that $\|T(x_0)\| = 1$ and $\|x_0 - x\| < \eps$. 
	
	\vspace{0.2cm}

\item[(b)] We say that $(X \times Y, Z)$ satisfies the {\bf L}$_{o,o}$ for bilinear mappings if given $\eps > 0$ and $B \in \mathcal{B}(X \times Y, Z)$ with $\|B\| = 1$, there exists $\eta(\eps, B) > 0$ such that whenever $(x, y) \in S_X \times S_Y$ satisfies $\|B(x, y)\| > 1 - \eta(\eps, B)$, there is $(x_0, y_0) \in S_X \times S_Y$ such that $\|B(x_0, y_0)\| = 1$, $\|x - x_0\| < \eps$, and $\|y - y_0\| < \eps$.
\end{itemize}
\end{definition}

\begin{example} \label{exam:L_{o,o}} Let us highlight some examples and results related to the properties just defined. 
	
	\begin{itemize}
		\item[(a)] If $\dim(X), \dim(Y) < \infty$, then $(X \times Y, Z)$ has the {\bf L}$_{o,o}$ for every Banach space $Z$ (see \cite[Proposition 2.2]{dklm}).
		\vspace{0.1cm} 
		\item[(b)] $(X \times Y, \K)$ has the {\bf L}$_{o,o}$ for bilinear mappings if and only if $(X, Y^*)$ has the {\bf L}$_{o,o}$ for operators, whenever $Y$ is uniformly convex (see \cite[Lemma 2.6]{dklm}). In particular, if $X$ is finite dimensional and $Y$ is uniformly convex, then $(X \times Y, \K)$ has the {\bf L}$_{o,o}$ for bilinear forms (see \cite[Theorem 2.4]{D}).
		\vspace{0.1cm} 
		\item[(c)] If $1 < p, q < \infty$, then $(\ell_p \times \ell_q, \K)$ has the {\bf L}$_{o,o}$ if and only if $p > q'$, where $q'$ is the conjugate of $q$ (see \cite[Theorem 2.7.(b)]{dklm}).
		\vspace{0.1cm} 
		\item[(d)] There are reflexive spaces $X$ and $Y$ such that $(X \times Y, \K)$ fails the {\bf L}$_{o,o}$ (see \cite[Theorem 2.21.(ii)]{D}).
	\end{itemize}
\end{example}

Our next aim is to prove that every nuclear operator between finite dimensional Banach spaces can be approximated by nuclear operators which attain their nuclear norm. This will follow from a more general result.

\begin{proposition}\label{theo:Loonucleadense}
	Let $X, Y$ be Banach spaces. Suppose that $(X^* \times Y, \mathbb{K})$ has {\bf L}$_{o,o}$ for bilinear forms. Then, every nuclear operator from $X$ into $Y$ can be approximated (in the nuclear norm) by nuclear operators which attain their nuclear norm. In other words, 
\begin{equation*} 
\overline{\NA_{\mathcal{N}}(X, Y)}^{\|\cdot\|_N} = \mathcal{N}(X, Y). 
\end{equation*} 
\end{proposition}

We get the following particular case by combining Proposition \ref{theo:Loonucleadense} with Example \ref{exam:L_{o,o}}.

\begin{corollary} \label{coro:finitedim} Let $X$ be finite dimensional Banach space. If $Y$ is uniformly convex, then 
\begin{equation*} 
\overline{\NA_{\mathcal{N}} (X, Y)}^{\|\cdot\|_N} = \mathcal{N}(X, Y).
\end{equation*} 
\end{corollary}

Now, we prove Proposition \ref{theo:Loonucleadense}.

\begin{proof}[Proof of Proposition \ref{theo:Loonucleadense}] Let $T \in \mathcal{N} (X, Y)$ and $\eps >0$ be given. There exists $H \in \mathcal{N} (X, Y)^*$ with $\| H \| = 1$ such that $H (T) = \|T\|_N$. Consider $G := (\delta^{-1} \circ \widetilde{J}^*)(H) \in (\ker J)^\perp$ (see Subsection \ref{remark:identinucleproj}). Let $A_G$ be the bilinear form on $X^* \times Y$ defined by $A_G (x^*, y) = G(x^*)(y)$ for every $x^* \in X^*$ and $y \in Y$. Then $\|A_G \| = \|G \| = 1$. Consider the positive value $\eta(\eps, A_G) >0$ from the assumption that $(X^* \times Y, \mathbb{K})$ has {\bf L}$_{o,o}$ for bilinear forms. Now, choose $(\lambda_n)_n \subseteq \mathbb{R}^+$, $(x_n^*)_n \subseteq S_{X^*}$, and $(y_n)_n \subseteq S_Y$ so that $T = \sum_{n=1}^\infty \lambda_n x_n^* \otimes y_n$ with 
\begin{equation*} 
\sum_{n=1}^\infty \lambda_n < \|T\|_N + \eta(\eps, A_G)^2.
\end{equation*} 
We get that
\begin{eqnarray*}
\|T \|_N = H(T) = \re H(T) &=& \sum_{n=1}^\infty \lambda_n \re \left( G(x_n^*)(y_n) \right) \\
&\leq& \sum_{n=1}^\infty \lambda_n | G(x_n^*)(y_n) | \\
&\leq& \sum_{n=1}^\infty \lambda_n < \|T\|_N + \eta(\eps, A_G)^2. \nonumber
\end{eqnarray*} 
In particular, 
\begin{equation}\label{estimation:sum_of_lambdas}
\sum_{n \in \N} \lambda_n \left( 1 - \re \left( G (x_n^*) (y_n) \right) \right) < \eta(\eps, A_G)^2
\end{equation}
Consider the following set
\begin{equation*} 
	I = \{ n \in \N : \re \left( G(x_n^*) (y_n) \right) > 1 - \eta( \eps, A_{G} ) \}.
\end{equation*} 
	From \eqref{estimation:sum_of_lambdas}, notice that 
	\[
	\eta( \eps, A_{G} ) \sum_{n \in I^c} \lambda_n \leq \sum_{n \in I^c} \lambda_n \left( 1 - \re \left( G (x_n^*) (y_n) \right) \right) < \eta( \eps, A_{G} )^2,
	\]
	which implies that $\sum_{ n \in I^c } \lambda_n < \eta(\eps, A_G)$. On the other hand, for each $n \in I $, 
	\[
	\re A_G(x_n^*, y_n) = \re \left( G (x_n^*) (y_n) \right) > 1 - \eta (\eps, A_{G}). 
	\]
	Thus, there exist norm one vectors $(\widetilde{x}_n^*)_{n \in I }$ in $X^*$ and $(\widetilde{y}_n)_{n \in I }$ in $Y$ such that 
	\[
	|A_G(\widetilde{x}_n^*, \widetilde{y}_n)| = |G (\widetilde{x}_n^*)(\widetilde{y}_n ) | = 1, \quad \| \widetilde{x}_n^* - x_n^* \| < \eps, \quad \text{and} \quad \| \widetilde{y}_n - y_n \| < \eps
	\]
	for every $n \in I $. Let us write $G(\widetilde{x}_n^*)( \widetilde{y}_n ) = e^{i \theta_n}$ with some $\theta_n \in \R$ for every $n\in I$. 
	Notice that $| 1 - e^{i\theta_n}| < \sqrt{2 \eta(\eps, A_G)}$ for every $n \in I$. Let us define 
	\[
	T' := \sum_{n \in I} \lambda_n e^{-i\theta_n} \widetilde{x}_n^* \otimes \widetilde{y}_n.
	\] 
	Then, 
	\begin{align*}
	\| T' - T \|_N &\leq \left\| \sum_{n \in I } \lambda_n ( e^{-i\theta_n} \widetilde{x}_n^* \otimes \widetilde{y}_n - x_n^* \otimes y_n )\right\|_N + \sum_{n \in I^c} \lambda_n \\
	&< \sum_{n \in I } \lambda_n | 1 -e^{i\theta_n} | + \left\| \sum_{n \in I } \lambda_n (\widetilde{x}_n^* \otimes \widetilde{y}_n - x_n^* \otimes y_n )\right\|_N + \eta(\eps, A_G) \\
	&< \sqrt{2 \eta(\eps, A_G)} (\|T\|_N + \eta(\eps, A_G)^2) + 2\eps ( \|T\|_N + \eta(\eps, A_G)^2) + \eta(\eps, A_G) \\
	&= (\sqrt{2 \eta(\eps, A_G)} + 2\eps) ( \|T\|_N + \eta(\eps, A_G)^2) + \eta(\eps, A_G). 
	\end{align*} 
Finally, it is clear by definition that $\|T'\|_N \leq \sum_{i \in I } \lambda_n$. On the other hand, 
	\begin{align*}
	\|T' \|_N \geq |H (T')| &= \left| \sum_{n\in I } \lambda_n e^{-i\theta_n} G(\widetilde{x}_n^*)(\widetilde{y}_n) \right|
	= \sum_{n\in I } \lambda_n. 
	\end{align*} 
	This shows that $T'$ attains its nuclear norm and completes the proof.
\end{proof}

Using very similar arguments to Proposition \ref{theo:Loonucleadense} and Corollary \ref{coro:finitedim}, we can obtain the following results.

\begin{proposition} \label{theo:Lootensordense} 	Let $X, Y$ be Banach spaces. Suppose that $(X \times Y, \mathbb{K})$ has {\bf L}$_{o,o}$ for bilinear forms. Then, every tensor in $X \pten Y$ can be approximated by tensors which attain their projective norm. In other words, 
\begin{equation*} 	
\overline{\NA_{\pi}(X\pten Y)}^{\|\cdot\|_{\pi}} = X\pten Y. 
\end{equation*} 
\end{proposition} 

\begin{corollary} \label{coro:finitedimtensor} Let $X$ be a finite dimensional Banach space. If $Y$ is uniformly convex, then 
\begin{equation*} 
\overline{\NA_{\pi} (X \pten Y)}^{\|\cdot\|_{\pi}} = X \pten Y.
\end{equation*} 
 \end{corollary}

Let us notice that, although we have the first examples of denseness by using Propositions \ref{theo:Loonucleadense} and \ref{theo:Lootensordense}, property {\bf L}$_{o,o}$ seems to be very restrictive. Indeed, when a pair of Banach spaces satisfies this property, both of them must be reflexive since every bilinear mapping attains its norm. Moreover, there are reflexive spaces $X$ and $Y$ such that $(X \times Y, \K)$ fails this property (see Example \ref{exam:L_{o,o}}.(d)). On the other hand, we could have used the previous results together with Example \ref{exam:L_{o,o}}.(c) in order to get examples where the denseness holds for $\ell_p$-spaces: for instance, if $1 < p,q < \infty$ and $p > q'$, then the set $\NA_{\pi}(\ell_p \pten \ell_q)$ is dense in $\ell_p \pten \ell_q$ by Proposition \ref{theo:Lootensordense}. Nevertheless, in what follows we will take advantage of the finite dimensional case to obtain more general examples of Banach spaces where the density follows. The only problem here is the fact that in general the projective norm does not respect subspaces, but it does respect $1$-complemented subspaces. For this reason, intuitively, we need a property of Banach spaces which guarantees the existence of many $1$-complemented subspaces. Motivated by this, we consider the following definition.

\begin{definition} \label{def:propertyP} Let $X$ be a Banach space. We will say that $X$ has the {\it metric $\pi$-property} if given $\varepsilon>0$ and $\{x_1,\ldots, x_n\}\subseteq S_X$ a finite collection in the sphere, then we can find a finite dimensional 1-complemented subspace $M\subseteq X$ such that for each $i \in \{1, \ldots, n \}$ there exists $x_i'\in M$ with $\Vert x_i-x_i'\Vert<\varepsilon$.
\end{definition}

Before proceeding, let us make a small observation. Let $\eps > 0$ and $F = \{x_1, \ldots, x_n\} \subseteq S_X$ be given. Suppose that $X$ has metric $\pi$-property as defined above and let $M$ be a finite dimensional subspace of $X$ with $\|x_i' - x_i\| < \eps$ for $x_i' \in M$ and $i = 1, \ldots, n$. Let $P_{\eps, F}$ be the norm one projection onto $M$. Then, for each $i = 1, \ldots, n$, we have
\begin{equation*}
\|P_{\eps, F} (x_i) - x_i\| \leq \|P_{\eps, F}(x_i) - P_{\eps, F} (x_i')\| + \|P_{\eps, F}(x_i') - x_i\| < 2 \eps.
\end{equation*}
Consider now the net $\left\{P_{\eps, F}: \eps > 0, \ F \subset S_X \ \mbox{a finite set} \right\}$ with $(\eps_1, F_1) \leq (\eps_2, F_2)$ if and only if $\eps_2 < \eps_1$ and $F_1 \subseteq F_2$. Then, $(P_{\eps, F})_{(\eps, F)}$ strongly converges to the identity on $S_X$ and hence on $X$ with $\|P_{\eps, F}\| \leq 1$ for every $\eps$ and $F$. This shows that Definition \ref{def:propertyP} is in fact equivalent to \cite[Definition 5.1]{CASAZZA} as the classical way of defining the metric $\pi$-property (we also send the reader to \cite{JRZ} and \cite{Lindenstrauss} for more information on the $\pi$-property). 

We have the following general result, which confirms that our intuition of finding a property of Banach spaces, which guarantees the existence of many $1$-complemented subspaces, was in the right direction. This result will give us many positive examples of denseness in both norm-attaining tensor and nuclear operator cases (see Examples \ref{examples:propertyP}). 

\begin{theorem} \label{theo:propertyP} Let $X$ be a Banach space satisfying the metric $\pi$-property.
\begin{itemize}
	\item[(a)] If $Y$ satisfies the metric $\pi$-property, then $\overline{\NA_{\pi}(X \pten Y)}^{\|\cdot\|_{\pi}} = X \pten Y$.
	\vspace{0.1cm}
	\item[(b)] If $Y$ is uniformly convex, then $\overline{\NA_{\pi}(X \pten Y)}^{\|\cdot\|_{\pi}} = X \pten Y$.
\end{itemize}	
 
\end{theorem}

\begin{proof} (a). Let $u\in S_{X \pten Y}$ and $\varepsilon>0$ be given. By \cite[Proposition 2.8]{rya}, there are bounded sequences $(\lambda_n)_n\subseteq \mathbb R^+$, $(x_n)_n\subseteq S_X$, and $(y_n)_n\subseteq S_Y$ with $u=\sum_{n=1}^\infty \lambda_n x_n\otimes y_n$ and 
\begin{equation} \label{lambdas}
\sum_{n=1}^\infty \lambda_n < 1+\eps.
\end{equation}	
Find $k \in \mathbb{N}$ large enough so that $\| u-z\|_{\pi_{X\pten Y}}<\frac{\varepsilon}{2}$ for $z:=\sum_{n=1}^k \lambda_n x_n\otimes y_n$. Since $X$ and $Y$ have the metric $\pi$-property, we can find finite dimensional subspaces $X_0$ of $X$ and $Y_0$ of $Y$ which are $1$-complemented and such that, for every $n\in\{1,\ldots, k\}$, there are $x_n'\in X_0$ and $y_n'\in Y_0$ such that 
 \begin{equation*} 
 \max \left\{ \Vert x_n-x_n'\Vert, \Vert y_n-y_n'\Vert \right\} < \frac{\varepsilon}{4k\lambda_n}.
 \end{equation*}

 Define $z'=\sum_{n=1}^k \lambda_n x_n'\otimes y_n'$ and notice that $\|z'-z\|_{\pi_{X\pten Y}} < \frac{\varepsilon}{2}$. Moreover, note that $z'\in X_0\otimes Y_0$. We have that $X_0$ is $1$-complemented in $X$ and $Y_0$ is $1$-complemented in $Y$. Consequently, by \cite[Proposition 2.4]{rya} we get that norm of $X\pten Y$ agrees on $X_0\otimes Y_0$ with the norm of $X_0\pten Y_0$. In particular, 
 \begin{equation} \label{norm-coincide}
 \Vert z'\Vert_{\pi_{X_0\pten Y_0}}=\Vert z'\Vert_{\pi_{X\pten Y}}. 
\end{equation} 	
Finally, since $X_0, Y_0$ are finite dimensional spaces, we use Proposition \ref{fin-dim-tomas} to show that $z'$ attains its projective norm in $X_0 \pten Y_0$. Since (\ref{norm-coincide}) holds, $z'$ attains its norm in $X \pten Y$ and we are done.

\noindent
(b). Let $u \in S_{X \pten Y}$ and $\eps > 0$ be given. There are bounded sequences $(\lambda_n)_n \subseteq \R^+$, $(x_n)_n \subseteq S_X$, and $(y_n)_n \subseteq S_Y$ with $u = \sum_{n=1}^{\infty} \lambda_n x_n \otimes y_n$ and (\ref{lambdas}) holds. We can find $k$ large enough such that $\|u - z\|_{\pi_{X\pten Y}} < \frac{\eps}{3}$ for $z := \sum_{n=1}^k \lambda_n x_n \otimes y_n$. Since $X$ satisfies the metric $\pi$-property, we can find a finite dimensional subspace $X_0$ which is 1-complemented and such that for every $n \in \{1, \ldots, k\}$, there is $x_n' \in X_0$ such that $\|x_n - x_n'\| < \frac{\eps}{6k \lambda_n}$. Define $z' = \sum_{n=1}^k \lambda_n x_n' \otimes y_n$. Notice that $\|z' - z\|_{\pi_{X\pten Y}} < \frac{\eps}{3}$ and that $z' \in X_0 \otimes Y$. Since $X_0$ is finite dimensional and $Y$ is uniformly convex, by Corollary \ref{coro:finitedimtensor}, we can find $z'' \in X_0 \pten Y$ such that
\begin{equation*}
\|z' - z''\|_{\pi_{X_0 \pten Y}} < \frac{\eps}{3} \ \ \ \mbox{with} \ \ \ z'' = \sum_{n=1}^{\infty} a_n \otimes b_n \ \ \ \mbox{and} \ \ \ \|z''\|_{\pi_{X_0 \pten Y}}= \sum_{n=1}^\infty \|a_n\| \|b_n\|.
\end{equation*}
Since the norm of $X\pten Y$ agrees on $X_0\otimes Y$ with the norm of $X_0\pten Y$, the result follows as in the previous item.
\end{proof}

Let us notice that if a Banach space $Z$ satisfies the metric $\pi$-property, then it has the metric approximation property and then the analogous result for nuclear operators follows immediately from Theorem \ref{theo:propertyP} and \cite[Corollary 4.8]{rya}.

\begin{corollary} \label{theo:propertyPnuclear} Let $X$ be Banach space such that $X^*$ satisfies the metric $\pi$-property. 
\begin{itemize}
	\item[(a)] If $Y$ satisfies the metric $\pi$-property, then $\overline{\NA_{\mathcal{N}}(X, Y)}^{\|\cdot\|_N} = \mathcal{N}(X, Y)$.
	\item[(b)] If $Y$ is uniformly convex, then $\overline{\NA_{\mathcal{N}}(X, Y)}^{\|\cdot\|_N} = \mathcal{N}(X, Y)$.
\end{itemize}
\end{corollary}

To finish this section, let us see particular cases where Theorem \ref{theo:propertyP} and Corollary \ref{theo:propertyPnuclear} can be applied. This shows that we always have denseness in all classical Banach spaces. Note that item (a) tells us that the metric $\pi$-property happens very often. Also, the stability results, (d), (e), (f), and (g), allow us to get more positive examples on denseness. We will first recall the following definition.

\begin{definition}\label{def:FDD}
Let $X$ be a Banach space. A sequence $\left\{ X_n\right\}_{n\in \mathbb{N}}$ of finite dimensional subspaces of $X$ is called a \textit{finite dimensional decomposition} of $X$ (F.D.D. for short) if every $x\in X$ has a unique representation of the form $x=\sum_{n=1}^{+\infty}x_n$ with $x_n\in X_n$ for every $n\in \mathbb{N}$. 
\end{definition}

\begin{remark}\label{rmk:FDD}
A F.D.D. on a Banach space $X$ determines a sequence $\{ P_n\}_{n\in \mathbb{N}}$ of projections (called the \textit{partial sum projections} of the decomposition) such that if $x=\sum_{n=1}^{\infty} x_n\in X$, then $P_j(x) = \sum_{n=1}^{j} x_n$ for all $j\in \mathbb{N}$. These projections are commuting, have increasing range, and converge strongly to the identity operator on $X$. The supremum of the norms of those projections is finite and is called the \textit{decomposition constant}.
\end{remark}

\begin{example} \label{examples:propertyP} The following Banach spaces satisfy the metric $\pi$-property (which might be well-known for some readers but we could not find proper references and we include the proof for completeness).
\begin{itemize}
\item[(a)] Banach spaces with a finite dimensional decomposition with the decomposition constant $1$ (consequently, every Banach space with Schauder basis can be renormed to have the metric $\pi$-property);
\vspace{0.2cm} 
\item[(b)] 	$L_p(\mu)$-spaces for any $1 \leq p < \infty$ and any measure $\mu$;
\vspace{0.2cm}
\item[(c)] $L_1$-predual spaces;
\vspace{0.2cm}
	\item[(d)] $X \oplus_a Y$, whenever $X, Y$ satisfy the metric $\pi$-property and $|\cdot|_a$ is an absolute norm;
	\vspace{0.2cm}
	\item[(e)] $X = \left[ \bigoplus_{n \in \N} X_n \right]_{c_0}$ or $\left[ \bigoplus_{n \in \N} X_n \right]_{\ell_p}$, $\forall$ $1 \leq p < \infty$, $X_n$ satisfying the metric $\pi$-property, $\forall n$;
	\vspace{0.2cm}
	\item[(f)] $X\pten Y$, whenever $X, Y$ satisfy the metric $\pi$-property;
	\vspace{0.2cm}
	\item[(g)] $X\iten Y$, whenever $X, Y$ satisfy the metric $\pi$-property. 
\end{itemize}
\end{example}

\begin{proof} (a). Given a Banach space $X$, if there exists a sequence of finite dimensional Banach spaces and $1$-complemented subspaces $\{E_n\}_{n\in\mathbb N}$ such that $E_n\subseteq E_{n+1}$ holds for every $n$ and such that $\bigcup_{n\in \mathbb N} E_n$ is dense in $X$, then $X$ has the metric $\pi$-property. In particular, it applies whenever $X$ is a Banach space with an F.D.D. with the decomposition constant $1$ (if $P_n:X\longrightarrow X$ are the associated norm-one projections, take $E_n:=P_n(X)$). 
 
\noindent
(b). Let $1\leq p<\infty$ be given. Then, $L_p(\mu)$ has the metric $\pi$-property regardless the measure $\mu$. Let us write $X = L_p(\mu)$, for short. Consider $x_1,\ldots, x_n\in S_X, \varepsilon>0$. For every $i\in\{1,\ldots, n\}$, we can find a simple function $x_i'\in S_X$ such that
\begin{equation}\label{L1aproxsimple}
\Vert x_i-x_i'\Vert<\eps,
\end{equation}	
where $x_i'=\sum_{j=1}^{m} a_{ij}\chi_{A_{j}}$ for suitable $m\in\mathbb N$, $a_{ij}\in\mathbb R$ and pairwise disjoint $A_{j}\in \Sigma$. Now, in order to prove that $X$ has the metric $\pi$-property, define $M:=\operatorname{span}\{\chi_{A_j}: 1\leq j\leq m\}$ and let us construct $P:X\longrightarrow X$ by the equation
\begin{equation*} 
T(f):=\sum_{j=1}^m \frac{1}{\mu(A_j)}\int_{A_j} f\ d\mu \chi_{A_j}.
\end{equation*} 
It is clear from the disjointedness of $A_1,\ldots, A_m$ and the fact that $\Vert P(f)\Vert\leq \Vert f\Vert$ holds for every $f\in X$. Furthermore, it is clear from the definition that $P(f)=f$ holds for every $f\in M$, so $P$ is a norm-one projection onto $M$. The result follows since $x_i'\in M$ and by the arbitrariness of $\eps>0$. This proves (b). 

\noindent
(c). If $X$ is a Banach space with $X^* = L_1$, then $X$ has the metric $\pi$-property. Indeed, let $\eps >0$ and $\{x_1, \ldots, x_n\} \subset S_X$ be given. Define $F_1 = \{0\}$ and $F_2 = \spann \{ x_1, \ldots, x_n\}$. By \cite[Theorem 3.1]{LL} and \cite[Theorem 1.3]{NO}, we may find a subspace $E$ of $X$ such that $E$ is isometric to $\ell_\infty^m$ for some $m \in \N$ and $d(x, E) < \eps$ for all $x \in F_2$. For each $ 1 \leq i \leq n$, pick $x_i' \in E$ so that $\| x_i - x_i' \| < \eps$. By \cite[Lemma 2.1]{MP}, there exists a norm one projection $P$ from $X$ to $E$; hence $E$ is indeed an $1$-complemented finite dimensional subspace of $X$.

\noindent
(d). To prove that the metric $\pi$-property is stable by absolute sums, let us first notice that $S_X$, in its definition, can be replaced by $B_X$ (indeed, let $\eps >0$ and $\{x_1, \ldots, x_n\} \subset B_X$ be given; without loss of generality, we may assume that $x_i \neq 0$ for all $1 \leq i \leq n$; from the metric $\pi$-property, we may find a 1-complemented finite dimensional space $M$ of $X$ with $x_i' \in M$ such that $\| x_i / \|x_i \| - x_i' \| < \eps$ for every $1\leq i \leq n$; thus, $\|x_i - \|x_i\| x_i' \| < \eps$ and $\{ \|x_1\| x_1', \ldots, \|x_n\| x_n' \} \subset M$). Set $Z = X \oplus_a Y$. Let $\eps >0$ and $\{ z_1, \ldots, z_n\} \subset S_Z$ be given. If we write $z_i = (x_i, y_i)$ for each $ 1 \leq i \leq n$, then $\max \{ \|x_i \|, \|y_i\| \} \leq \|z_i \|_a = 1$ for every $ 1 \leq i \leq n$. As $X$ has the metric $\pi$-property and $\{x_1, \ldots, x_n \} \subset B_X$, there exist a $1$-complemented finite dimensional subspace $M$ of $X$ and $\{ x_1', \ldots, x_n' \} \subseteq M$ such that $\|x_i - x_i' \| < \eps$. Similarly, there exist a $1$-complemented finite dimensional subspace $N$ of $Y$ and $\{ y_1', \ldots, y_n' \} \subset N$ such that $\|y_i - y_i' \| < \eps$. If we let $z_i' = (x_i', y_i')$ for each $1 \leq i \leq n$, then for every $1 \leq i \leq n$, we have
\begin{equation*} 
\|z_i - z_i' \|_a \leq \|x_i - x_i'\| + \| y_i - y_i' \| < 2\eps.
\end{equation*} 
Let $P$ and $Q$ be norm one projections from $X$ onto $M$ and $Y$ onto $N$, respectively. Consider the map $(P,Q)$ defined on $X \oplus_a Y$ as $(x,y) \mapsto (Px, Qy)$ for every $(x,y) \in X \oplus_a Y$. Note that $(P,Q)$ is a projection with (closed) range $M \oplus_a N$. Moreover, 
\[
\| (Px , Qy) \|_a = | (\| Px \|, \|Qy \| ) |_a \leq | ( \|x\|, \|y\|) |_a = \| (x, y) \|_a 
\]
for every $(x, y) \in X \oplus_a Y$; hence $M \oplus_a N$ is a 1-complemented finite dimensional subspace of $Z$ with $\{ z_1' , \ldots, z_n' \} \subset M \oplus_a N$ satisfying $\|z_i - z_i' \| < 2\eps$ for each $ 1 \leq i \leq n$. 

\noindent
(e). This can be obtained by extending the proof of (d). Let $\{x_1,\ldots ,x_n\} \subseteq S_X$ be given. First, approximate $x_i$ by $x_i'$ of finite support. Now, say $x_i'=(x_{i1},\ldots , x_{ik},0,0,\ldots)$ with some common $k \in \mathbb{N}$. Find a 1-complemented subspace $M_j$ in $X_j$ containing ${x_{1j}, \ldots, x_{nj}}$ from the assumption that $X_j$ enjoys the metric $\pi$-property for each $1 \leq j \leq k$. Then, $M=\{(z_1,z_2,\ldots ,z_k,0,0, \ldots): z_i\in M_i, 1 \leq i \leq k\}$ is a finite dimensional subspace of $X$ which is 1-complemented by the projection $(P_1,P_2,\ldots,P_k,0,0,\ldots)$ (defined similarly as in the item (d)) and $M$ contains the set $\{ x_1', \ldots, x_n'\}$.

\noindent
(f). Let $\eps > 0$ and $z_1, \ldots, z_n \in S_{X \pten Y}$ be given. For each $1 \leq i \leq n$, consider $\{ x_j^{(i)}, y_j^{(i)} \} \subseteq B_X \times B_Y$ to be such that
\begin{equation*}
z_i = \sum_{j=1}^{\infty} x_j^{(i)} \otimes y_j^{(i)} \ \ \ \mbox{with} \ \ \ \|z_i\|_{\pi} > \sum_{j=1}^{\infty} \| x_j^{(i)} \| \| y_j^{(i)} \| - \eps. 
\end{equation*}
For each $i=1, \ldots, n$, let $N_i \in \N$ be such that
\begin{equation*}
\sum_{j=N_i+1}^{\infty} \| x_j^{(i)} \| \| y_j^{(i)} \| < \frac{\eps}{2}.
\end{equation*}
Now, since $X$ has the metric $\pi$-property, there exists a 1-complemented finite dimensional subspace $M$ of $X$ with
\begin{equation*}
\left\{ \widetilde{x_j}^{(i)}: 1 \leq j \leq N_i, 1 \leq i \leq n \right\} \subseteq M \ \mbox{such that} \ \| \widetilde{x_j}^{(i)} - x_j^{(i)} \| < \min \left\{ \frac{\eps}{4N_i}: 1 \leq i \leq n \right\}
\end{equation*}
and, analogously, there exists a $1$-complemented finite dimensional subspace $N$ of $Y$ with
\begin{equation*}
\left\{ \widetilde{y_j}^{(i)}: 1 \leq j \leq N_i, 1 \leq i \leq n \right\} \subseteq N \ \mbox{such that} \ \| \widetilde{y_j}^{(i)} - y_j^{(i)} \| < \min \left\{ \frac{\eps}{4N_i}: 1 \leq i \leq n \right\}
\end{equation*}
for each $1 \leq j \leq N_i$ with $i=1,\ldots, n$. By \cite[Proposition 2.4]{rya}, $M \pten N$ is an $1$-complemented space. Let $\widetilde{z_i} := \sum_{j=1}^{N_i} \widetilde{x_j}^{(i)} \otimes \widetilde{y_j}^{(i)}$. Then,
\begin{equation*}
\left\| \widetilde{z_i} - \sum_{j=1}^{N_i} x_j^{(i)} \otimes y_j^{(i)} \right\|_{\pi} \leq 2 N_i \min \left\{ \frac{\eps}{4N_i}: 1 \leq i \leq n \right\} \leq \frac{\eps}{2} 
\end{equation*}
for every $i = 1, \ldots, n$. Then, $X \pten Y$ has the metric $\pi$-property, as desired.

\noindent
(g).  Let $z_1, \ldots, z_n \in S_{X \iten Y}$ and $\delta > 0$ be given. For each $i \in \{1, \ldots, n\}$, let $\widetilde{z_i} \in X \otimes Y$ be such that $\|z_i - \widetilde{z_i}\|_{\eps} < \frac{\delta}{2}$. Let $\sum_{j=1}^{N_i} x_j^{(i)} \otimes y_j^{(i)}$ be a representation of $\widetilde{z_i}$ for each $i = 1, \ldots, n$. Since
	\begin{equation*}
	\{x_j^{(i)}: 1 \leq j \leq N_i, 1 \leq i \leq n \} \subseteq X \ \ \ \mbox{and} \ \ \ \{y_j^{(i)}: 1 \leq j \leq N_i, 1 \leq i \leq n \} \subseteq Y,
	\end{equation*}
	there are $1$-complemented finite dimensional subspaces $M \leq X$ and $N \leq Y$ with $\{\widetilde{x_j}^{(i)}: 1 \leq j \leq N_i, 1 \leq i \leq n \} \subseteq M$ and $\{\widetilde{y_j}^{(i)}: 1 \leq j \leq N_i, 1 \leq i \leq n \} \subseteq N$ such that
	\begin{equation*}
	\|x_j^{(i)} - \widetilde{x_j}^{(i)}\| < \min \left\{ \frac{\eps}{4 N_i}: 1 \leq i \leq n \right\} \ \mbox{and} \ \|y_j^{(i)} - \widetilde{y_j}^{(i)}\| < \min \left\{ \frac{\eps}{4 N_i}: 1 \leq i \leq n \right\}.
	\end{equation*}
	As $M \iten N$ is a 1-complemented subspace of $X \iten Y$ (see, for instance, \cite[Proposition 3.2]{rya}), 
	\begin{equation*}
	\widetilde{v_i} = \sum_{j=1}^{N_i} \widetilde{x_j}^{(i)} \otimes \widetilde{y_j}^{(i)} \in M \iten N \ \ \ \mbox{and} \ \ \ \|\widetilde{z_i} - \widetilde{v_i}\|_{\eps} \leq \|\widetilde{z_i} - \widetilde{v_i}\|_{\pi} \leq \frac{\delta}{2},
	\end{equation*}
	which implies that $\|z_i - \widetilde{v_i}\|_{\eps} < \delta$, we have that $X \iten Y$ satisfies the metric $\pi$-property.
 \end{proof} 

\begin{remark}\label{remark:tensorazo}
From the estimates of case (g) above it follows that $X\widehat{\otimes}_\alpha Y$ has the metric $\pi$-property whenever $X$ and $Y$ enjoy the metric $\pi$-property and $\alpha$ is a \textit{uniform cross norm} (see \cite[Section 6.1]{rya} for background and details).
\end{remark}

 

Example \ref{examples:propertyP}.(g) allows us to extend Theorem \ref{theo:propertyP} for larger projective tensor products. 

\begin{corollary} Let $N \in \N$ be given. Let $X_1, \ldots, X_N$ be Banach spaces with the metric $\pi$-property, and $Y$ be a Banach space. Then,
\begin{equation*} 
\overline{\NA_{\pi}(X_1 \pten \cdots \pten X_N \pten Y)}^{\|\cdot\|_{\pi}} = X_1 \pten \cdots \pten X_N \pten Y.
\end{equation*}  
\end{corollary}

\section{There are tensors which cannot be approximated by norm-attaining tensors} \label{contraejemplocojonudo}

By the results from previous section, one may think that the denseness for norm-attaining tensors always holds true. In this section, we will see that this is not the case. We show that there are Banach spaces $X$ and $Y$ such that the set of all tensors in $X \pten Y^*$ which attain their projective norms {\it is not} dense in $X \pten Y^*$. In order to do that, let us notice that, by Theorem \ref{theo:charprojattain}, it would be enough to show that $\NA(X,Y^{**})\cap B_{L(X,Y^{**})}$ is not norming for $X\pten Y^*$ (and in fact that is what we do; see Remark \ref{remark:NAnormante}). On the other hand, in view of the proof of \cite[Proposition 2.3]{llr2}, note that if either $X$ or $Y^{**}$ satisfies the metric approximation property (respectively, bounded approximation property), then $\mathcal{F}(X,Y^{**})$ is norming (respectively, $K$-norming) for $X\pten Y^{*}$, and this implies that $\mathcal F(X,Y^{**})$ is $w^*$-dense in $\mathcal{L}(X,Y^{**})$. This suggests us to look for our counterexample in the context of Banach spaces failing the approximation property and trying to guarantee that the set of operators which attain their norms is not bigger than the set of finite-rank operators. This is the reason why we will adapt \cite[Theorem 1]{martinjfa} taking into account all the previous considerations.

For this, we will use Read's space $\mathcal{R}$ (see \cite{KGM, KGMW, R}  for all the details on this space). Read's space is a renorming of the Banach space $c_0$, $\mathcal{R} = (c_0, \nn \cdot )$, which has bidual $\mathcal{R}^{**}$ strictly convex (see \cite[Theorem 4]{KGM}). This implies that $\NA(X, \mathcal{R}^{**}) \subseteq \mathcal{F}(X, \mathcal{R}^{**})$ whenever $X$ is a closed subspace of $c_0$ (see \cite[Lemma 2]{martinjfa}). It is worth mentioning that we are not using here the deep properties of $\mathcal{R}$ (that it contains no two-codimensional proximal subspaces) but only the fact that its bidual is strictly convex for the bidual norm and that it contains $c_0$ (this is in fact well-known; the existence of such norms can be justified, for instance, by using \cite[Lemma 2.1]{KGMW} and taking $R$ as a one-to-one operator from $c_0$ into $\ell_2$).

\begin{theorem} \label{counterexample} Let $\mathcal{R}$ be Read's space. There exist subspaces $X$ of $c_0$ and $Y$ of $\mathcal{R}$ such that the set of tensors in $ X \pten Y^*$ which attain their projective norms is not dense in $X \pten Y^*$. 
\end{theorem}

In order to prove Theorem \ref{counterexample}, we would like to present several results, which, from our point of view, have their own interest.

\begin{lemma} \label{counter:lemma1} Let $X, Y$ be a Banach spaces such that $Y^*$ is separable. If $\mathcal{F} (X, Y^{**})$ is viewed as a subspace of $(X \pten Y^*)^* = \mathcal{L}(X, Y^{**})$, we have 
\begin{equation*} 
	B_{\mathcal{F}(X, Y^{**})} \subset \overline{B_{\mathcal{F}(X,Y)}}^{w^*}. 
\end{equation*} 
\end{lemma} 

\begin{proof} 
Let $T \in {\mathcal{F}(X, Y^{**}) }$ with $\|T \| < 1$. Choose a countable dense subset $(y_n^*)_{n \in \N}$ of $Y^*$ and let $F_n = \text{span} \{ y_1^*, \ldots, y_n^* \}$ for each $n \in \N$. By the Principle of Local Reflexivity, for each $n \in \N$, there exists an operator $\phi_n : T(X) \rightarrow Y$ such that 
\begin{enumerate}
\item $\left( 1- \frac{1}{n} \right) \|T(x)\| \leq \| \phi_n (T(x)) \| \leq \left( 1+\frac{1}{n} \right) \|T(x)\|$ for every $x \in X$, 
\vspace{0.1cm}
\item $y^* (\phi_n (T(x) )) = y^* (T(x))$ for every $y^* \in F_n$ and $x \in X$.  
\end{enumerate} 
Choose $n_0 \in \N$ so that $\frac{1}{n} < \frac{1}{\|T\|} - 1$ whenever $n \geq n_0$. Let us define $K_n = \phi_n \circ T \in \mathcal{F} (X, Y)$ for each $n \geq n_0$. Then $\| K_n \| \leq \|\phi_n \| \| T \| < 1$ for each $n \geq n_0$. 
We claim that $K_n \xrightarrow{w^*} T$. First, observe that given $x \in X$ and $m \in \N$, we have 
\begin{equation}\label{PLR_estimate}
y_m^* (K_n (x)) = y_m^* ( \phi_n ( T(x))) = y_m^* ( T(x)) \text{ for every $n \geq m$.}
\end{equation}
Now, let $x \in X \setminus \{0\} $, $y^* \in Y^*$ and $\eps >0$ be given. Pick $n_0 \in \N$ so that $\| y_{n_0}^* - y^* \| < \frac{\eps}{2\|x\|}$. By \eqref{PLR_estimate}, we have for $n \geq n_0$, 
\begin{align*}
| y^* (K_n (x))  - y^* ( T(x)) | &\leq | y^* (K_n (x)) -y_{n_0}^* ( K_n (x)) | + |y_{n_0}^* ( K_n (x))  -  y_{n_0}^* ( T(x)) | \\
&\hspace{6cm}+ | y_{n_0}^* ( T(x))  - y^* ( T(x)) | \\
&\leq \|y^* - y_{n_0}^* \| \|K_n \| \|x\| + \| y_{n_0}^* - y^* \| \|T\|\|x\| \\
&< \frac{\eps}{2} + \frac{\eps}{2}  = \eps.
\end{align*} 
By a linearity argument we get that $K_n(z)\rightarrow T(z)$ for every $z\in X\otimes Y$. Finally, since the sequence $K_n$ is bounded we get that $K_n\rightarrow T$ in the $w^*$-topology.

This implies that $\{ T \in \mathcal{F} (X , Y^{**} ) : \| T \| < 1 \} \subset \overline{B_{\mathcal{F}(X,Y)}}^{w^*}$. As a $w^*$-closed set in $\mathcal{L} (X, Y^{**})$ is $\|\cdot\|$-closed, we conclude that 
$B_{\mathcal{F}(X, Y^{**})} \subset \overline{B_{\mathcal{F}(X,Y)}}^{w^*}$. 
\end{proof} 



In what follows, we will be using the {\it strong operator topology} ($SOT$, for short) and the {\it weak operator topology} ($WOT$, for short). Recall that the strong operator topology in $\mathcal{L}(X, Y)$ is the topology defined by the basic neighborhoods
\begin{equation*}
N(T; A, \eps) = \left\{S \in \mathcal{L}(X, Y): \|(T - S)(x)\| < \eps, x \in A \right\},
\end{equation*}
where $A$ is an arbitrary finite subset of $X$ and $\eps > 0$. Thus, in the $SOT$, a net $(T_{\alpha})$ converges to $T$ if and only if $(T_{\alpha}(x))$ converges to $T(x)$ for every $x \in X$. On the other hand, the weak operator topology is defined by the basic neighborhoods 
\begin{equation*}
N(T; A, A^*, \eps) = \left\{ S \in \mathcal{L}(X, Y), |y^*(T - S)(x)| < \eps, y^* \in A^*, x \in A \right\},
\end{equation*}
where $A$ and $A^*$ are arbitrary finite sets in $X$ and $Y^*$, respectively, and $\eps > 0$. Thus, in the $WOT$, a net $T_{\alpha}$ converges to $T$ if and only if $(y^*(T_{\alpha}(x)))$ converges to $y^*(T(x))$ for every $x \in X$ and $y^* \in Y^*$.

Let us notice that a convex set in $\mathcal{L}(X, Y)$ has the same closure in the $WOT$ as it does in the $SOT$ (see, for instance, \cite[Corollary 5, page 477]{DS}). We will use this fact in the proof of Theorem \ref{counterexample} below.

\begin{lemma}\label{lemma:APSOT} Let $X$ be a Banach space failing the approximation property. Then, the identity map on $X$ does not belong to $R \overline{ B_{\mathcal{F}(X, X) }}^{WOT}$ for any $R > 0$. 
\end{lemma} 

\begin{proof} Let $X$ be a Banach space which fails the approximation property and let us denote the identity map on $X$ by $\id_X$. Then, by definition, $\id_X \not\in \overline{ \mathcal{F} (X, X)}^{\,\tau}$, where $\,\tau$ is the topology of uniform convergence on compact sets. For given $R > 0$, let us prove that $\id_X \not\in R \overline{B_{\mathcal{F}(X, X)}}^{SOT}$. In order to get a contradiction, let us assume $\id_X \in R \overline{B_{\mathcal{F}(X, X)}}^{SOT}$. Then there exists a net $(T_\alpha)_{\alpha \in \Lambda} \subset R B_{\mathcal{F}(X,X) }$ such that $T_\alpha \xrightarrow{SOT} \id_X$. Now, let $K$ be a compact set in $X$ and $\eps > 0$ be given. Choose a $\left( \min \left\{\frac{\eps}{3R}, \frac{\eps}{3} \right\}\right)$-net $\{x_1, \ldots, x_k\}$ for $K$. Pick $\alpha_0 \in \Lambda$ such that for every $\alpha \geq \alpha_0$ 
\begin{equation*} 
\max_{1 \leq i \leq k} \|T_\alpha (x_i) - \id_X (x_i) \| = \max_{1 \leq i \leq k} \|T_\alpha (x_i) - x_i \| < \frac{\eps}{3}. 
\end{equation*} 
Given $x \in K$, take $i \in \{1,\ldots, k\}$ so that $\|x - x_i \| < \min \left\{\frac{\eps}{3R}, \frac{\eps}{3} \right\}$. Then, 
	\begin{align*}
	\|  T_\alpha (x) - \id_X (x) \|  &\leq \|T_\alpha (x) - T_\alpha (x_i)\| + \|T_\alpha (x_i) - x_i \| +  \|x_i - x \|  \\
	&\leq   \|T_\alpha \| \|x - x_i \| + \frac{\eps}{3} + \frac{\eps}{3} \\
	&< \frac{\eps}{3} + \frac{\eps}{3} + \frac{\eps}{3} = \eps
	\end{align*} 
	for every $\alpha \geq \alpha_0$. This implies that $\id_X \in \overline{\mathcal F (X, X)}^{\,\tau}$, a contradiction. So, $\id_X \not\in R \overline{ B_{\mathcal{F}(X, X)}}^{SOT} = R\overline{ B_{\mathcal{F}(X, X)}}^{WOT}$.
\end{proof}

Now we are ready to prove Theorem \ref{counterexample}.

\begin{proof}[Proof of Theorem \ref{counterexample}] Let $X$ be a closed subspace of $c_0$ which fails the approximation property (see, for instance, \cite[Theorem 2.d.6]{LT}). Then, by Lemma \ref{lemma:APSOT}, the identity map on $X$ does not belong to $R \overline{ B_{\mathcal{F}(X, X) }}^{WOT}$ for any $R > 0$. Let $Y = (X, \nn \cdot)$, where $\nn \cdot$ is the norm that defines Read's space. Let us denote by $\iota \in {\mathcal{L} (X, Y)}$ the formal identity map from $X$ to $Y$. Then $T=\iota / \| \iota \|$ does not belong to $R \overline{ B_{\mathcal{F}(X, Y) }}^{WOT}$ for any $R > 0$. It follows that $T$ does not belong to $R \overline{ B_{\mathcal{F}(X, Y) }}^{w^*}$, where the previous weak-star topology refers to $\sigma(\mathcal{L}(X, Y^{**}), X \pten Y^*)$, for any $R > 0$. Indeed, if $T \in R \overline{ B_{\mathcal{F}(X, Y) }}^{w^*}$ for some $R > 0$, given $x \in X$, $y^* \in Y^*$ and $\eps >0$, there exists $T_0 \in R B_{\mathcal{F}(X, Y)}$ such that 
	\[
	\left| y^* ( T(x) - T_0 (x)) \right| = \left| (T- T_0) ( x \otimes y^*) \right| < \eps,
	\]
which implies that $T \in R \overline{B_{\mathcal{F} (X, Y)}}^{WOT}$, a contradiction. 
In particular, $T$ does not belong to $\overline{ B_{\mathcal{F}(X, Y) }}^{w^*}$. As $Y^*$ is separable, by Lemma \ref{counter:lemma1}, $T$ does not belong to $\overline{B_{\mathcal{F}(X, Y^{**})}}^{w^*}$. 
Thus, by the Hahn-Banach theorem we have that the  unit ball $B_{\mathcal{F} (X, Y^{**})}$ is not norming for $X \pten Y^*$. Take $z \in X \pten Y^*$ with $\|z\|_{\pi} = 1$ and $\alpha >0$ such that 
	\begin{equation}\label{alpha}
	\sup \{ |G(z)| : G \in B_{\mathcal{F}(X, Y^{**})} \} < 1 - \alpha.
	\end{equation} 
	
\noindent	 \vspace{0.1cm}
{\bf Claim}: $\displaystyle \dist \left(z, \NA_{\pi} (X \pten Y^*) \right) > \frac{\alpha}{2}$. 

If this is not the case, there exists $z' \in \NA_{\pi} (X \pten Y^*)$ such that $\|z-z'\|_{\pi} \leq \frac{\alpha}{2}$. This implies that $\|z'\|_{\pi} \geq 1 - \frac{\alpha}{2}$. Let $G \in \mathcal{L} (X, Y^{**})$ with $\|G \| = 1$ such that $|G(z')| = \|z'\|_{\pi}$. In particular, $G\in \NA(X,Y^{**})$ by Theorem \ref{theo:charprojattain}. Notice that $Y^{**} = Y^{\perp\perp}$ is a closed subspace of $\mathcal{R}^{**}$, so $Y^{**}$ is strictly convex. Thus, we have that $G \in \mathcal{F}(X, Y^{**})$ by \cite[Lemma 2]{martinjfa}, which implies by \eqref{alpha} that $|G(z) | < 1- \alpha$. Nevertheless,
\begin{equation*} 
	|G(z)| \geq |G(z')| - \|z-z'\|_{\pi} \geq 1- \frac{\alpha}{2} - \frac{\alpha}{2} = 1 - \alpha,
\end{equation*} 
which is a contradiction.
\end{proof}

\begin{remark}\label{remark:NAnormante}
Notice that from the above proof it follows that, given two Banach spaces $X$ and $Y$, if $\NA_{\pi}(X\pten Y)$ is dense in $X\pten Y$, then $\NA(X,Y^*)\cap B_{\mathcal L(X,Y^*)}$ is norming for $X\pten Y$. 
\end{remark}

In fact, from the proof of Theorem \ref{counterexample} (and its lemmas) we extract more information. Recall that for every non-zero tensor $u \in X \otimes Y$, there is a smallest $N \in \N$ for which there is a representation for $z$ containing $N$ terms. The number $N$ is known as the rank of $u$. Because of this, we will say that $u$ is a {\it finite-rank tensor} if $u \in X\otimes Y$. Although it is not known whether every finite-rank operator can be approximated by norm-attaining operators, the case for tensors does not hold in general.

\begin{proposition} There are tensors of finite-rank which do not attain their projective norm.
\end{proposition}

\begin{proof} Consider $X$ and $Y^*$ as in Theorem \ref{counterexample}. Then, there exist $\alpha > 0$ and $z \in X \pten Y^*$ such that $\dist (z, \NA_{\pi}(X \pten Y^*)) \geq \alpha$. Now, take $u$ of finite-rank such that $\|z - u\|_{\pi} < \frac{\alpha}{2}$. Then, this element cannot attain its projective norm.
\end{proof}

As we have commented at the beginning of this section, let us notice that from the proof of Theorem \ref{counterexample}, there exist some Banach spaces $X$ and $Y$ such that $\NA(X, Y^{**}) \cap B_{\mathcal{L}(X, Y^{**})}$ is not $w^*$-dense in $ B_{\mathcal{L}(X, Y^{**})}$. Actually, we have the following result.

\begin{corollary} \label{theo:wstartdensity} There are Banach spaces $X$ and $Y$ such that
	\begin{equation*} 	
	\overline{\co(\NA(X, Y^{**}) \cap B_{\mathcal{L} (X, Y^{**})}) }^{w^*} \not= B_{\mathcal{L}(X, Y^{**})}.
	\end{equation*} 
\end{corollary}

\section{Open Questions}
 
In this section, we would like to discuss and propose some open questions.

We  have proved that if $H$ is a complex Hilbert space, every tensor in $H \pten H$ attains its projective norm (see Proposition \ref{prop:hilbertnuclearallNA}) and that the set $\NA_{\pi}(  L_p (\mu) \pten L_q (\nu) )$ is dense in $ L_p (\mu) \pten L_p (\nu)$ for $1 < p,q < \infty$ and measures $\mu$ and $\nu$ (see Example \ref{examples:propertyP}.(b)). However, we do not know what happens in general when both factors are reflexive spaces. 

\begin{question} Let $X, Y$ be reflexive Banach spaces. Is it true that the set of all norm-attaining tensors is dense in $X \pten Y$?
\end{question} 

Let us notice that, by trying to mimic the proof of Theorem \ref{counterexample} (and its lemmas) for the nuclear operator case, one would realize that
 \begin{equation*}\label{thm:condition_nuclear}
 (\ker J )^\perp \neq \overline{ (\ker J)^\perp \cap F ( Y , X^{**})}^{w^*}  
 \end{equation*}
needs to be one the hypothesis (which we cannot guarantee that it holds). We do not know if there is a version of Theorem \ref{counterexample} for nuclear operators.

\begin{question} Are there Banach spaces $X$ and $Y$ so that $\NA_{\mathcal{N}}(X,Y)$ is not dense in $\mathcal N(X,Y)$?
\end{question}

We say that a Banach space $X$ has {\it property quasi-}$\alpha$ if, for an index set $\Gamma$, there are $A = \{x_{\gamma} \in S_X: \gamma \in \Gamma\}$, $A^* = \{x_{\gamma}^* \in S_{X^*}: \gamma \in \Gamma \}$, and $\lambda: A \longrightarrow \R$ such that $x_{\gamma}^*(x_{\gamma}) = 1$ for every $\gamma \in \Gamma$; $|x_{\gamma}^* (x_{\eta})| \leq \lambda(x_{\gamma}) < 1$ for $\gamma \not= \eta$; and for every $e \in \Ext(B_{X^{**}})$, there is a subset $A_e \subseteq A$ and a scalar $t$ with $|t| = 1$ such that $te \in \overline{Q(A_e)}^{w^*}$ and $r_e = \sup \{\lambda(x): x \in A_e\} < 1$, where $Q$ is the canonical embedding on $X^{**}$ (see \cite{CS}). Let us notice that property quasi-$\alpha$ is weaker than property $\alpha$ introduced by W. Schachermayer in \cite{schaalp}. We have proved that $\NA_{\pi}(\ell_1 \pten Y) = \ell_1 \pten Y$ for every Banach space $Y$ (see Proposition \ref{exam:nucopec0attain}). Consequently, by using Proposition \ref{prop:allprojattain}, we get that 
\begin{equation*}
\overline{\NA(\ell_1 \times Y)}^{\|\cdot\|} = \mathcal{B}(\ell_1 \times Y)
\end{equation*}
for every Banach space $Y$. This is a particular case of \cite[Theorem 2.17]{CS}, which we wonder if it could be extended in the following sense.

\begin{question} Let $X$ be a Banach space with property $\alpha$ (or quasi-$\alpha$). Is it true that the equality $\NA_{\pi}(X \pten Y) = X \pten Y$ holds for every Banach space $Y$?
\end{question}

\textbf{Acknowledgements:} The authors are grateful to William B. Johnson for pointing our an error in the proof of their original version of Lemma \ref{lemma:APSOT}. They are also grateful to Gilles Godefroy, Petr H\'ajek, and Tommaso Russo for giving proper references for the metric $\pi$-property. They also would like to thank Richard Aron, Gin\'es L\'opez P\'erez, and Miguel Mart\'in for fruitful conversations on the topic of the paper. Finally, the authors want to thank the anonymus referee for his/her valuable mathematical contributions and the linguistic suggestions which have highly improved the quality and the exposition of the paper.

\end{document}